\documentclass[12pt]{amsart}

\usepackage{amsmath}
\usepackage{amssymb} 
\usepackage{array}
\usepackage{enumitem}
\usepackage{hyperref}
\usepackage[all]{xy}
\usepackage{verbatim} 
\usepackage{color}
\definecolor{battleshipgrey}{rgb}{0.52, 0.52, 0.51} 
\hypersetup{
       colorlinks=true,       
    linkcolor=battleshipgrey,     
    citecolor=battleshipgrey,        
    filecolor=magenta,      
    urlcolor=battleshipgrey          
}
\usepackage{pstricks-add}

\theoremstyle{plain}
\newtheorem{theorem}{Theorem}[section]
\newtheorem{lemma}[theorem]{Lemma}

\newtheorem{definition}[theorem]{Definition}
\theoremstyle{remark}
\newtheorem{remark}{Remark}[section]
\newtheorem{example}{Example}[section]
\newtheorem*{notation}{Notation}
\newtheorem*{acknowledgment}{Acknowledgment}

\numberwithin{equation}{section}

\newcommand{\bA}{\mathbb{A}}

\newcommand{\K}{\mathbb{K}}

\newcommand{\R}{\mathbb{R}}
\newcommand{\Z}{\mathbb{Z}}
\newcommand{\N}{\mathbb{N}}

\newcommand{\cP}{{\mathcal P}}

\newcommand{\g}{\mathfrak{g}}

\newcommand{\ttt}{\mathbf{t}}

\newcommand{\vvv}{\mathbf{v}}

\newcommand{\pG}{\mathsf{PG}}

\newcommand{\sfn}{\mathsf{n}}

\newcommand{\sfone}{\mathsf{1}}
\newcommand{\sftwo}{\mathsf{2}}

\newcommand{\Diff}{\mathrm{Diff}}

\newcommand{\opp}{\mathrm{opp}}

\newcommand{\Sec}{\mathrm{Sec}}
\newcommand{\Bis}{\mathrm{Bis}}

\newcommand{\im}{\mathrm{im}}

\newcommand{\id}{\mathrm{id}}
\newcommand{\fin}{\mathrm{fin}}

\newcommand{\pr}{\mathrm{pr}}

\newcommand{\Bij}{\mathrm{Bij}}

\newcommand{\sett}[1]{{\{ #1 \}}}
\newcommand{\settt}[1]{{\overline{\{ #1 \}}}}

\newcommand{\Bigsetof}[2]{\begin{Bmatrix} #1 \,\Big|\, #2 \end{Bmatrix}}

\newcommand{\ssk}{\smallskip}
\newcommand{\nin}{\noindent}

\newcommand{\ol}{\overline}

\begin{document}

\title[Lie Calculus]{Lie Calculus} 

\author{Wolfgang Bertram}

\address{Institut \'{E}lie Cartan de Lorraine \\
Universit\'{e} de Lorraine at Nancy, CNRS, INRIA \\
B.P. 70239 \\
F-54506 Vand\oe{}uvre-l\`{e}s-Nancy Cedex, France}

\email{\url{wolfgang.bertram@univ-lorraine.fr}}



\begin{abstract} 
We explain that general differential calculus and Lie theory have a common foundation:
{\em Lie Calculus} is differential calculus, seen from the point of view of Lie theory, by making use of the
{\em groupoid} concept as link between them. Higher order theory naturally involves
{\em higher algebra} ($n$-fold  groupoids). 
 \end{abstract}

\maketitle

\setcounter{tocdepth}{1}

\section*{Introduction}

When working on the foundations of differential calculus (in chronological order, 
\cite{BGN04, Be08, Be13, Be15, Be15b}), I got the impression that there ought to exist a comprehensive algebraic 
theory, englobing both the fundamental results of calculus and of differential
geometry, and where {\em Lie theory} is a kind of
Ariadne's thread. Confirming this impression, {\em groupoids} turned out, in my most recent approach
\cite{Be15, Be15b}, to be the most remarkable algebraic structure underlying calculus. 
These groupoids are in fact {\em Lie groupoids}, and Lie theoretical features can be used even before starting
to develop Lie theory properly. In this sense, Lie theory and the development of ``conceptual'' calculus go hand in
hand, whence the term ``Lie Calculus'' chosen here. 
There are many similarities with the approach by {\em synthetic differential geometry}\footnote{ cf.\ \cite{Ko10,
MR91};  see Subsection \ref{ssec:SDG} at the end of this paper.},
and, of course, with the ideas 
present in Charles Ehresmann's \oe uvre (cf.\ \cite{KPRW07} for an overview): in a sense, I simply propose
to apply his ideas not only to differential geometry, but already to calculus itself. 
The reader certainly realizes that this sounds like a big program, and indeed the present short text, though entirely 
self-contained, is far from giving
a final and complete exposition of these ideas. I hope to have time and occasion to develop them in more length and
depth in some not too distant future. 

\ssk
Lie Calculus, as understood here,  can be cast in  three formulae.
We consider functions $f:U \to W$, where
$U$ is an (open) subset in a $\K$-vector space 
$V$.  The first formula defines the {\em first extended domain} of $U$:
\begin{equation}\label{1}
U^{[1]} := \{ (x,v,t) \in V \times V \times \K \mid \, x \in U, x+tv \in U \} .
\end{equation}
The second formula goes with Theorem \ref{th:tangentgroupoid} saying
that the pair of sets 
\begin{equation}\label{2}
U^\sett{1}:= (U^{[1]} , U \times \K) ,
\end{equation}
with source $\alpha$ and target $\beta$, units, 
 product and inversion defined as in the theorem,
is a {\em groupoid}.
The third formula describes the ``iteration'' of (\ref{2}): one would like to define the ``double extension'' by
$(U^\sett{1})^\sett{1}$, but since it turns out that one has to remember the order in which these iterated
extensions are performed, we must first make a formal copy 
$\sett{k}$ of  the symbol $\sett{1}$, for each $k \in \N$, and then define
\begin{equation}\label{3}
U^\sfn := U^\sett{1,2,\ldots,n} := (\ldots(U^\sett{1})^\sett{2} \ldots )^\sett{n} . 
\end{equation}
Then (Theorem \ref{th:Usfn}) $U^\sfn$ is an 
{\em $n$-fold groupoid}, called the {\em $n$-fold tangent groupoid} of $U$
(def.\ \ref{def:n-fold}; indeed, it is a higher order generalization of {\em Connes' tangent groupoid},
{\it cf.} def.\ \ref{def:tangentgroupoid}). 
A map $f:U \to W$ then is {\em smooth} if, and only if, it has natural prolongations to groupoid morphisms
$f^\sfn:U^\sfn \to W^\sfn$, for all $n \in \N$ (Theorem \ref{th:tangent-n}).
Studying the structure of $f^\sfn$ and the one of $U^\sfn$ go hand in hand. 

\ssk
A first aim of the present text is to make these three formulae intelligible: to give the necessary background and 
definitions, and to indicate the (elementary) proofs. 
A second aim is to unfold them a little bit more: to
 give some ideas about their consequences and about what kind of theory emerges from them.
As said above, the full unfolding will be a matter for another book. 
 
\ssk
Here is a short description of the contents of this work:
Basic notions and ideas on {\em groupoids} are presented in
Section \ref{sec:groupoids}.  
In Section \ref{sec:calculus}, we explain  that
first order calculus of a map $f$ is described by groupoids, via formulae (\ref{1}) and (\ref{2}).
We also establish the 
 {\em chain rule} 
$(g \circ f)^\sfn=g^\sfn \circ f^\sfn$.
The chain rule is the basic tool needed to define {\em atlasses and manifolds}.
In the present approach, speaking about
manifolds is less essential than in the usual presentation, and the corresponding
Section \ref{sec:manifolds} is rather short. Indeed, our constructions are natural from the very outset, and
hence it is more or less obvious that everything carries over to the manifold level: the groupoid
$M^\sett{1}$ is an intrinsic object associated to any (Hausdorff) manifold $M$.  
The step from first order to higher order calculus is, conceptually, most important and challenging: already in
usual calculus, the procedure of {\em iterating} is not quite straightforward, and in the present
approach, it naturally leads to higher, {\em $n$-fold groupoids}. A (hopefully) simple and down-to-earth
presentation of this concept is given in Section \ref{sec:doublegroupoids}.
With this preparation at hand, Sections \ref{sec:firstLie} and \ref{sec:higherorder} are the heart of the present work:
(general) higher order calculus works by using several times principles of (first order) Lie calculus.
We concentrate on the {\em symmetric cubic} theory, and show that it can be understood from the point of view of scalar extension by
{\em cubes of rings} (Theorem \ref{th:scalarextension}).
These definitions are the beginning of a far-reaching theory whose full exposition would need more space.
In order to give an impression of its possible scope, at the end of this paper we give some more 
comments on Lie Theory (subsection \ref{ssec:Lietheory}), Connection Theory (subsection \ref{ssec:connections}),
and on further problems (section \ref{sec:perspectives}) such as
the case of {\em discrete} base rings, ``full'' cubic calculus and the scaloid,  relation with SDG,
and the case of possibly {\em non-commutative} base rings and supercalculus. 

\begin{notation}
For $n \in \N$, the {\em standard $n$-element set} is denoted by
\begin{equation}\label{eqn:sfn}
\sfn := \{ 1,\ldots ,n \} . 
\end{equation}
\end{notation}

\begin{acknowledgment}
The present work has been presented at the 
\href{http://50sls.impan.pl//slides/Bertram.pdf}{50th Seminar Sophus Lie in B\c edlewo, September 2016}, 
and I would like to thank the organizers for inviting me and for the great job they did in organizing this conference.
I also thank the unknown referee for helpful comments.  
\end{acknowledgment}

\section{Groups, and their cousins}\label{sec:groupoids}

In Lie Theory, but also in general mathematics, {\em groups} play a double r\^ole: on the one hand, they are an 
{\em object} of study
in their own right, and on the other hand, they are an important {\em tool}, or even: a  {\em part of mathematical
language}, used for studying a great variety of topics. 
This double aspect is shared by some of  their 
``cousins''. Recall that a group has a {\em binary, everywhere defined, and associative} 
product, {\em one} unit, and {\em inversion}. Then,
\begin{itemize}
\item
forgetting the unit but
keeping an everywhere defined product we get {\em torsors},
\item
forgetting associativity, but keeping one unit and invertibility, we get {\em loops},
\item
allowing {\em many units}, and a {\em not everwhere defined product}, we get {\em groupoids},
\item
forgetting inversion in a groupoid, we get {\em small categories},
\item
forgetting the units in a groupoid, we get {\em pregroupoids}.
\end{itemize}
\nin
In this work, we will not talk about {\em loops}, although, via the theory of {\em connections}, they have
a close relation to the topics to be discussed here (see subsection \ref{ssec:connections}).

\subsection{Groups without unit: torsors} \label{ssec:groups}
We start with a {\em group}. But
sometimes one wishes to get rid of its unit element, just like affine spaces are sometimes preferable to vector spaces.
A simple and efficient way to describe this procedure algebraically is to replace the {\em binary} product map by
the {\em ternary} product map $G^3 \to G$, $(x,y,z) \mapsto (xyz):=xy^{-1}z$. It satisfies the algebraic identities

\ssk
(IP) idempotency:  $(xyy)=x=(yyx)$,

(PA) para-associativity:  $((uvw)yz) = (uv(wyz)) = (u (ywv) z)$.
\ssk

\nin
By definition, a {\em torsor} is a set together with an everywhere defined ternary map satisfying (IP) and
(PA).\footnote{There is no really standard terminolgy: other terms are
{\em heap, groud, principal homogeneous space}... Using the term ``torsor'' in our sense
has been popularized by John Baez.}
 It is easy to prove that every torsor $M$, after fixing an element $y\in M$, becomes a group with product
$xz:=(xyz)$.  The converse is also true:
 torsors are for groups what affine spaces are for vector spaces (folklore).

\subsection{Groupoids}
By now, it is widely realized that {\em groupoids} are omnipresent in mathematics -- see
 \cite{Br87, CW99, Ma05, W96}. Since there are various definitions and conventions, it is important to fix one throughout
 a given text. Here is our's:

\begin{definition}\label{def:groupoid}
A {\em groupoid} $G =(G_1,G_0,\alpha,\beta,\ast,1,i)$ is given by:
 a set $G_0$ of {\em objects}, a set $G_1$ of {\em morphisms}, by
{\em source} and {\em targent maps} $\alpha,\beta: G_1 \to G_0$, a product $\ast$ defined on the set
$$
G_1 \times_{\alpha,\beta} G_1 :=
\bigl\{ (a,b) \in G_1 \times G_1 \mid \alpha(a)=\beta(b) \bigr\} ,
$$ 
such that $\alpha(a \ast b)=\alpha(b)$ and  $\beta(a\ast b)=\beta(a)$ and 
$(a \ast b) \ast c= a\ast (b \ast c)$ whenever $\beta(c) =\alpha(b)$ and $\beta(b)=\alpha(a)$; 
a {\em unit section} $1:G_0 \to G_1$, $x \mapsto 1_x$ such that
$\alpha \circ 1 = \id_{G_0} = \beta \circ 1$, and
$a \ast 1_{\alpha(a)} = a$, $1_{\beta(b)} \ast b = b$, 
and an 
{\em inversion map} $i:G \to G$, $a \mapsto a^{-1}$ such that
$a \ast a^{-1} = 1_{\beta(a)}$, $a^{-1} \ast a = 1_{\alpha(a)}$.
\end{definition}

\nin Following \cite{CW99, W96}, we shall represent a groupoid by drawing its morphism set.   
Fibers of $\alpha$ and $\beta$ are represented by  grey lines whose directions are given by the two arrows,
labelled $\alpha,\beta$, and
the object set $G_0$ is identified with the image of the unit section (fat horizontal line in the figure).

\begin{figure}[h]
\caption{Representation of a groupoid}
\begin{center}
\psset{xunit=0.6cm,yunit=0.3cm,algebraic=true,dotstyle=o,dotsize=3pt 0,linewidth=0.8pt,arrowsize=3pt 2,arrowinset=0.25}
\begin{pspicture*}(-4.3,-5.0)(19.08,4.3)
\psline{->}(-2,1)(1,-2)
\psline{->}(-2,1)(-2,-2)
\psline[linewidth=2.8pt](3.48,-2.58)(13.14,-2.52)
\psplot[linecolor=lightgray]{-4.3}{29.08}{(--2.7-3*x)/3}
\psplot[linecolor=lightgray]{-4.3}{29.08}{(--31.86-3*x)/3}
\psline[linecolor=lightgray](13.14,-10.06)(13.14,6.3)
\psline[linecolor=lightgray](3.48,-10.06)(3.48,6.3)
\psplot[linecolor=lightgray]{-4.3}{29.08}{(--7.47-3*x)/3}
\psplot[linecolor=lightgray]{-4.3}{29.08}{(--41.23-9.66*x)/9.66}
\psplot[linecolor=lightgray]{-4.3}{29.08}{(--61.76-9.66*x)/9.66}
\psplot[linecolor=lightgray]{-4.3}{29.08}{(--78.29-9.66*x)/9.66}
\psline[linecolor=lightgray](5.06,-10.06)(5.06,6.3)
\psline[linecolor=lightgray](6.83,-10.06)(6.83,6.3)
\psline[linecolor=lightgray](8.94,-10.06)(8.94,6.3)
\psline[linecolor=lightgray](10.64,-10.06)(10.64,6.3)
\psline[linewidth=1.6pt](5.06,3.04)(5.06,1.33)
\psline[linewidth=1.6pt](5.06,1.33)(5.06,-0.79)
\psline[linewidth=1.6pt](5.06,3.04)(6.83,1.28)
\psline[linewidth=1.6pt](6.83,1.28)(8.94,-0.84)
\psline[linewidth=1.6pt](6.83,1.28)(6.83,-0.43)
\psline[linewidth=1.6pt](5.06,1.33)(6.83,-0.43)
\psline[linewidth=1.6pt](5.06,-0.79)(6.83,-2.56)
\psline[linewidth=1.6pt](6.83,-0.43)(6.83,-2.56)
\psline[linewidth=1.6pt](6.83,-0.43)(8.94,-2.55)
\psline[linewidth=1.6pt](8.94,-0.84)(8.94,-2.55)
\psline[linewidth=1.6pt](6.83,-2.56)(8.94,-4.67)
\psline[linewidth=1.6pt](8.94,-2.55)(8.94,-4.67)
\begin{scriptsize}
\rput[bl](-0.28,-0.14){$\beta$}
\rput[bl](-2.52,-0.54){$\alpha$}
\rput[bl](10.9,2.84){$G_1$ (morphism set)}
\rput[bl](10.9,-3.84){$G_0$ (object set)}
\psdots[dotstyle=*,linecolor=darkgray](5.06,-0.79)
\rput[bl](5.24,-0.66){\darkgray{$f$}}
\psdots[dotstyle=*,linecolor=darkgray](6.83,-0.43)
\rput[bl](7,-0.32){\darkgray{$g$}}
\psdots[dotstyle=*,linecolor=darkgray](5.06,1.33)
\rput[bl](3.14,1.46){\darkgray{$g\ast f$}}
\psdots[dotstyle=*,linecolor=darkgray](8.94,-0.84)
\rput[bl](9.22,-0.7){\darkgray{$h$}}
\psdots[dotstyle=*,linecolor=darkgray](5.06,3.04)
\rput[bl](5.14,3.26){\darkgray{$h \ast g \ast f$}}
\rput[bl](7.14,1.46){\darkgray{$h \ast g$}}
\psdots[dotstyle=*,linecolor=darkgray](6.83,1.28)
\psdots[dotstyle=*,linecolor=darkgray](8.94,-4.67)
\rput[bl](9.18,-4.64){\darkgray{$g^{-1}$}}
\end{scriptsize}
\end{pspicture*}
\end{center}
\end{figure}

\begin{example}[Pair groupoids]
For every set $M$, the {\em pair groupoid} $\pG(M)$ is defined by:
$G_1 = M\times M$, $G_0=M$, $\alpha(y,x)=x$, $\beta(y,x)=y$, $1_x = (x,x)$,
$(z,y) \ast (y,x)=(z,x)$, $(y,x)^{-1}=(x,y)$.
 In this case, one might rather be inclined to represent $G_0$ by a diagonal line, and
$\beta$ by horizontal lines. 
The pair $(y,x)$ may be seen as the ``zero jet'' of a function sending $x$ to $y$, and the pair groupoid
may thus be considered as the {\em groupoid of jets of order zero}.
\end{example}

\begin{example}\label{ex:equivalencerelation}
Let $\epsilon = \{ (x,y) \in M^2 \mid x \sim y \}$ be (the graph of) an equivalence relation $\sim$ on $M$.
Then $G_1 = \epsilon$, $G_0 = M$ defines a subgroupoid of the pair groupoid.
\end{example}

\begin{example}[Groups]
If $\alpha =\beta$, then every fiber $[y]_\alpha = \{ g \in G_1 \mid \alpha(g)=y \}$ is a group with unit $1_y$: we have
a {\em group bundle}. If, moreover, $G_0$ is a singleton, then $G_1$ is  a usual group.
Thus groupoids generalize groups. 
\end{example}

\subsection{Small cats}
By {\em small cat} we shall abbreviate the term
 {\em small category}: it is defined just like a groupoid, without requiring existence of the inverse $i$.
For instance, if in Example \ref{ex:equivalencerelation}, $\epsilon$ is reflexive and transitive, but not symmetric, we
get a small cat. 
A small cat with one object is a {\em monoid}.
A groupoid can be defined as a small cat in which every morphism is invertible.
When we use the word ``category'', we mean ``(possibly) big category'' (that is, the collection of objects and morphisms 
need not form a set in the sense of naive set theory). 

\subsection{Pregroupoids}
With groupoids, we may play the game described above, forgetting the units in order to get the groupoid analog of 
a torsor, called a  {\em pregroupoid}: 
we retain properties of the {\em ternary} product
$(abc):= a \ast b^{-1}\ast c$, defined on the set
$$
G_1 \times_{\alpha} G_1 \times_\beta G_1 :=
\bigl\{ (a,b,c) \in G_1 \times G_1 \times G_1  \mid \alpha(a)=\alpha(b), \beta(c)=\beta(c)  \bigr\}  .
$$ 
As is immediately checked, the ternary product satisfies idempotency (IP) and
para-associativity (PA) (see above, \ref{ssec:groups}).
A {\em pregroupoid} is defined to be a set $G_1$ with two surjections $\alpha:G_1 \to A$,
$\beta:G_1 \to B$ and a ternary product defined on
$G_1 \times_{\alpha} G_1 \times_\beta G_1$ satisfying these two properties
(definition due to Kock, cf.\ \cite{Be14b}).

\begin{example}
If  $A=B$ is a singleton, then a pregroupoid is the same as a {\em torsor}.
\end{example}

\begin{example}\label{ex:product}
Let $A,B$ sets, let $G_1 := B \times A$, $\alpha = \pr_A$ and $\beta = \pr_B$ the two projections, and
when $\alpha(a,b)=\alpha(a',b'), \beta(a'',b'')=\beta(a',b')$, i.e., $b=b'$, $a'=a''$,
$$
\bigl( (a,b),(a',b'),(a'',b'') \bigr) := (a,b'') \, .
$$
You may call this a ``pair-pregroupoid''. 
If $A=B$, this is the pair groupoid with $(uvw)=u\ast v^{-1} \ast w$, by forgetting the unit section; else it is ``new''. 
\end{example}

\subsection{Functors}\label{ssec:functors}
A  {\em functor}
between small cats or groupoids $G = (G_1,G_0)$ and $G' = (G_1',G_0')$ is given by  a pair of maps 
$f = (f_1:G_1 \to G_1', f_0:G_0 \to G_0')$ such  that
\begin{enumerate}
\item $\alpha' \circ f_1 = f_0 \circ \alpha$, $\quad \beta' \circ f_1 = f_0 \circ \beta$, $\quad 1' \circ f_0 = f_1 \circ 1$,
\item $\forall (a,b) \in G_1 \times_{\alpha,\beta} G_1$ : $f_1 (a \ast b) = f_1(a) \ast' f_1(b)$.
\end{enumerate} 
Obviously, small cats, and groupoids and their functors form (big!) categories.

\subsection{Opposites}
For each small cat or groupoid  $G$, there is an {\em opposite small cat (groupoid)} 
$G^\opp$, given by the same sets, and 
$\alpha^\opp := \beta$, $\beta^\opp = \alpha$, $a \ast^\opp b := b \ast a$,  $i^\opp = i$ and $1^\opp = 1$. 
A  {\em contravariant functor} is a functor into an opposite cat.

\subsection{Sections and bisections}\label{ssec:bisections}
An {\em $\alpha$-section} of $(G^1,G^0)$ is a subset $S \subset G^1$ which is a representative set for
$\alpha$-classes, and likewise for {\em $\beta$-sections}.
The spaces of such sections are denoted by
\begin{align}
\Sec_\alpha(G) & := \bigl\{
S \subset G^1 \mid \,
\forall x \in G_0 : \exists^! s=s(x) \in S : x =\alpha(s) \bigr\},
\\
\Sec_\beta(G) & := \bigl\{
S \subset G^1 \mid \,
\forall x \in G_0 : \exists^! s=s(x) \in S : x =\beta(s) \bigr\} .
\end{align}
Of course, then $S = \im(s)$ is uniquely determined by the map $s:G^0 \to G^1$, which is a section of $\alpha$, resp.\
of $\beta$.  
A {\em bisection} is a section both of $\alpha$ and of $\beta$, and the space of all bisections is denoted by
\begin{equation}
\Bis(G):= \Sec_\alpha(G) \cap \Sec_\beta(G).
\end{equation}

\nin
The proof of the following two theorems is straightforward (cf.\ \cite{CW99, Be14b}).

\begin{theorem}[Monoid of sections, group of bisections]
For every groupoid $G$,
the power set $\cP(G^1)$ forms a monoid with respect to the product $S \ast R$ induced by the groupoid law $\ast$
of $G$, and  unit $1=1_{G_0}$ the unit section, 
$$
S \ast R = \{  s \ast r \mid \, s \in S, r\in R, \alpha(s)=\beta(r) \} .
$$
The sets $\Sec_\alpha(G)$ and $\Sec_\beta(G)$ are sub-moinoids of
$\cP(G)$ such that $(\Sec_\alpha(G))^{-1} = \Sec_\beta(G)$. In particular, $\Bis(G)$ is a group, called the
{\em group of bisections} of $G$.
\end{theorem}

\begin{example}\label{ex:bisections}[Binary relations]
Let $G = \pG(M)$ be the pair groupoid of a set $M$.
Then $\cP(G^1) = \cP(M\times M)$ is the set of binary relations on $M$ with their usual relational product, and
$\Sec_\alpha(G)$ is the set of (graphs of) mappings $f:M\to M$, and $\Bis(G)=\Bij(M)$ the group of bijections
of $M$. Note that $\Sec_\beta(G)$ is the set of ``duals'' of mappings; there is no common word in mathematics
to name it. 
\end{example}

\begin{theorem}[Anchor]\label{th:anchor0}
For each groupoid $(G^1,G^0)$, the {\em anchor map} $(\Upsilon,\id_{G_0})$, 
$$
\Upsilon : G_1 \to G_0\times G_0,,\quad g \mapsto (\beta(g),\alpha(g)) ,
$$
is a functor from $G$ to $\pG(G_0)$,
and it induces a group morphism
$$
\Bis(G) \to \Bij(G_0), \quad
S \mapsto (x \mapsto \beta( S \cap [x]_\alpha)) \, .
$$
\end{theorem}

\begin{remark}
A groupoid is called {\em principal} if $\Upsilon$ is an isomorphism. 
This holds iff the groupoid is isomorphic to a pair groupoid. In this sense, principal groupoids ``are'' the pair
groupoids.
\end{remark}

\section{The groupoid of differential calculus} \label{sec:diffcal}\label{sec:calculus}

\subsection{The classes $C^n$}
Let us briefly review ``usual'' differential calculus. 
The crucial operation  is to take the limit $t\to 0$ in the {\em difference quotient} 
(\ref{eqn:differencequotient}) of a map
$f:U \to W$, where $f$ is defined on an (open) set $U$ in a vector space $V$, with values in another vector space $W$,
\begin{equation} \label{eqn:differencequotient}
f^{[1]} (x,v,t) := \frac{ f(x+tv) - f(x)}{t} \, .
\end{equation}
In other words, filling in the ``missing value'' for $t=0$, we can extend the difference quotient to a map 
$f^{[1]}:U^{[1]} \to W$ defined on the whole set $U^{[1]}$ given by (\ref{1}).
It is more or less folklore that  this map is {\em continuous} iff $f$ is of class $C^1$:

\begin{theorem}\label{th:difftheorem}
Assume $\K=\R, V=\R^n,W= \R^m$. The following are equivalent:
\begin{enumerate}
\item
$f$ is of class $C^1$,
\item
 the difference quotient map extends
to a {\em continuous} map
$f^{[1]} : U^{[1]} \to W$.
\end{enumerate}
Under these conditions, the differential of $f$ is given by
$df(x)v= f^{[1]}(x,v,0)$.
Moreover, with the same notation, the following are also equivalent:
\begin{enumerate}
\item[(1')]
$f$ is of class $C^n$,
\item[(2')]
$f$ is $C^1$, and $f^{[1]}:U^{[1]} \to W$ is of class $C^{n-1}$. 
\end{enumerate}
\end{theorem}

\nin
The proof is a nice exercise in undergraduate calculus -- see, e.g., \cite{Be08, Be11} for the solution, and \cite{BGN04}
for generalizations to various infinite dimensional situations. 
As observed in \cite{BGN04}, property (2') from the theorem can serve much more generally
as a {\em definition} of the class $C^n$ over non-discrete topological fields, or even more generally, over ``good''
topological rings:

\begin{definition}\label{def:C1K}\label{def:CnK}
Assume $\K$ is a
 {\em good topological ring}, meaning,   a topological ring whose unit group $\K^\times$ is dense in $\K$.
A map $f:U \to W$ from an open set $U$ in a topological $\K$-module $V$ to a a topological $\K$-module
$W$ is called {\em of class $C^1_\K$} if it satisfies property (2) from the preceding theorem, i.e., if a continuous
map $f^{[1]}:U^{[1]} \to W$, extending the difference quotient, exists.  
The class $C^n_\K$ is defined inductively by using property (2') from the theorem, and  the
{\em higher order extended domains} and {\em higher order difference quotient maps} are defined inductively by
\begin{align*}
U^{[n]} & := (U^{[n-1]})^{[1]}, \\
f^{[n]} & := (f^{[n-1]})^{[1]} : U^{[n]} \to W .
\end{align*}
\end{definition}

\nin
Calculus based on this definition, called {\em topological differential calculus},
 has excellent properties, which by the way clarify and simplify proofs of 
well-known facts from ``usual'' real calculus. One uses, over and over, the ``density principle'':

\begin{lemma}[Prolongation of identities]\label{la:density}
If $f$ is of class $C^n$, then all algebraic identities satisfied for $f^{[n]}$ and for {\em invertible} scalars in the arguments of
$f^{[n]}$ continue to hold, by continuity and density, for {\em all} scalars.
\end{lemma}

\begin{example}\label{ex:lin}
For instance, {\em linearity of the first differential} is obtained by this principle as follows: first, for invertible $t$, by
direct and trivial computation, 
\begin{align}\label{eqn:add}
f^{[1]}(x,v+v',t) & = f^{[1]}(x,v,t) + f^{[1]}(x+vt,v',t) .
\end{align}
By prolongation of identities, if $f$ is $C^1$,
 this also holds for $t=0$, whence additivity
$df(x)(v+v')=df(x)v + df(x)v'$.
Homogeneity is proved similarly (see \cite{BGN04}). 
Thus
 in topological differential calculus, linearity of the differential $df(x)$ is a {\em theorem},
in contrast to he traditional approach by Fr\'echet differentiability, where  it is an {\em assumption}. 
By the philosophical principle known as  Occam's razor, eliminating this assumption
can be considered as a methodological advantage of 
topological differential calculus, compared to the usual one. 
Put differently, the idea of considering differential calculus as a ``linearization machine''  is a consequence, and
not an an input, in our approach. In this respect, one might say that we are coming back to the original ideas of
Newton and Leibniz -- who rather thought in terms of ``continuity of nature'' than in terms of ``approximation of nature
by linear algebra''. 
\end{example}

\subsection{The tangent groupoid}
The most fundamental structure of $U^{[1]}$ is the one of a groupoid.
Topology is not needed in the following

\begin{theorem}[The groupoid $U^\sett{1}$]\label{th:tangentgroupoid}
Assume $V$ is a module over a ring $\K$, $U \subset V$ is non-empty, and define $U^{[1]}$ by
Eqn.\ (\ref{1}). Then the pair $(G_1,G_0) = (U^{[1]}, U \times \K)$, with projections and unit section defined by
$$
\alpha (x,v,t):= (x,t), \qquad
\beta(x,v,t):=(x+tv,t), \qquad
1_{(x,t)} := (x,0,t),
$$
and product $\ast$  and inverse $i$  given by (when $x'=x+tv$ and $t'=t$)
$$
(x',v',t') \ast (x,v,t) = (x,v'+v,t), \qquad
(x,v,t)^{-1} = (x-tv,-v,t),
$$
is a groupoid which we shall denote by $U^\sett{1}$. 
For each fixed value of $t$, the same formulae define a groupoid denoted by
$$
U^\sett{1}_t:=(U_t,U) := (\{ (x,v) \mid (x,v,t) \in U^{[1]} \} , U) \, .
$$
\end{theorem}

\begin{proof}
The properties from Definition \ref{def:groupoid} are checked by straightforward computation. We urge the reader
to check this (full details are given in \cite{Be15}).
For instance, let us here just prove the condition $\beta(a\ast b)=\beta(a)$:
$$
\beta(x',v+v',t)= x + t(v+v') =  (x+tv) + tv' = x' + tv' = \beta(x',v',t) .
$$
Since $t$ remains ``silent'' in these computations, $(U_t,U)$ is also a groupoid.
\end{proof}

\begin{theorem}[Anchor of $U^\sett{1}$]\label{th:anchor}
For invertible $t$, the groupoid $U_t$ is isomorphic to the pair groupoid of $U$, and for $t=0$, it is the tangent bundle of $U$.
More precisely, for each invertible scalar $t$, the {\em anchor map} 
$$
\Upsilon : U_t \to U \times U, \quad (x,v) \mapsto (\beta(x,v),\alpha(x,v)) = (x+tv,x) 
$$
defines an isomorphism $(\Upsilon,\id_U)$ between the
groupoid $U_t$  and the pair groupoid
$\pG(U)=(U \times U,U)$. For $t=0$, the groupoid $U_t$ is a group bundle, given by
$$
(TU,U) := ( U \times V,U), \quad \alpha(x,v)=x=\beta(x,v), \quad
(x,v) \ast (x,v') = (x,v+v') .
$$
\end{theorem}

\begin{proof}
Recall from th.\ \ref{th:anchor0} that $\Upsilon$ always defines a groupoid morphism. 
Let $t \in \K^\times$, the group of invertible scalars.
Then $\Upsilon$
is bijective, with inverse given by
$\Upsilon^{-1} (z,x) = (\frac{1}{t} (z-x),x)$.
When $t=0$, we get $\beta(x,v)=x+0 v=x=\alpha(x)$, so $\alpha = \beta$, and
 we have a group bundle as described in the theorem.
\end{proof}

\begin{definition}\label{def:tangentgroupoid}
The groupoid $U^\sett{1}$ is called the
{\em tangent groupoid\footnote{This terminology follows Connes \cite{Co94}, Section II.5,
where 
 in case $\K=\R$ and for $t \in [0,1]$ the tangent groupoid is defined by a disjoint union
$TU \cup (\pG (U)) \times ]0,1]$.  }   of $U$}.
The group bundle $(TU,U)$ is called the {\em tangent bundle of $U$}, and the groupoid
$$
U^\sett{1}_{\fin} := (U, \{ (x,v,t) \in U^{[1]} \mid t \in \K^\times \})   \, \cong \,   \pG(U) \times \K^\times
$$
is called the {\em finite part of the tangent groupoid}. 
Note that, if $\K$ is a field, then $U^\sett{1}$ is the disjoint union of $U^\sett{1}_\fin$ and $TU$.
\end{definition}

\nin
One should think of the family $(U_t)_{t\in \K}$ of groupoids as a sort of contraction of the pair groupoid ($t=1$) towards
the tangent bundle ($t=0$), by letting $\beta$-fibers become more and more vertical as $t$ tends to $0$, as in 
Figure 2. 
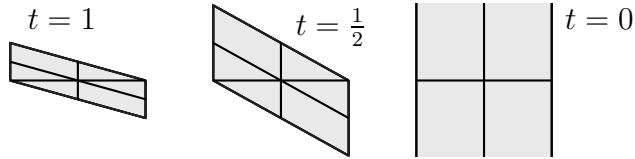
\begin{figure}[h] \caption{Tangent groupoid}
\begin{center} 
\newrgbcolor{sqsqsq}{0.13 0.13 0.13}
\psset{xunit=0.9cm,yunit=0.25cm,algebraic=true,dotstyle=o,dotsize=3pt 0,linewidth=0.8pt,arrowsize=3pt 2,arrowinset=0.25}
\begin{pspicture*}(-4.3,-4.06)(29.08,4.1)
\pspolygon[linecolor=sqsqsq,fillcolor=sqsqsq,fillstyle=solid,opacity=0.1](0,2)(0,0)(2,-2)(2,0)
\pspolygon[linecolor=sqsqsq,fillcolor=sqsqsq,fillstyle=solid,opacity=0.1](3,4)(5,0)(5,-4)(3,0)
\pspolygon[linecolor=sqsqsq,fillcolor=sqsqsq,fillstyle=solid,opacity=0.1](6,5)(8,5)(8,-5)(6,-5)
\psline(0,0)(0,2)
\psline(0,2)(2,0)
\psline(2,0)(2,-2)
\psline(2,-2)(0,0)
\psline(0,2)(0,0)
\psline[linecolor=sqsqsq](0,2)(0,0)
\psline[linecolor=sqsqsq](0,0)(2,-2)
\psline[linecolor=sqsqsq](2,-2)(2,0)
\psline[linecolor=sqsqsq](2,0)(0,2)
\psline(3,0)(3,4)
\psline(3,4)(5,0)
\psline(3,0)(5,-4)
\psline(5,0)(5,-4)
\psline(6,0)(8,0)
\psline(6,-10.06)(6,6.3)
\psline(8,-10.06)(8,6.3)
\psline[linecolor=sqsqsq](3,4)(5,0)
\psline[linecolor=sqsqsq](5,0)(5,-4)
\psline[linecolor=sqsqsq](5,-4)(3,0)
\psline[linecolor=sqsqsq](3,0)(3,4)
\psline(0,0)(2,0)
\psline(3,0)(5,0)
\psline(1,1)(1,-1)
\psline(0,1)(2,-1)
\psline(7,-10.06)(7,6.3)
\psline(3,2)(5,-2)
\psline(4,2)(4,-2)
\rput[tl](0.26,3.80){$t=1$}
\rput[tl](4.22,3.80){$t=\frac{1}{2}$}
\rput[tl](8.2,3.80){$ t=0$}
\end{pspicture*}
\end{center}
\end{figure}

\ssk \nin
Using a fixed scalar $s$, we can relate $U_t$ and $U_{st}$. In \cite{Be15}, this has been formalized into a 
{\em double category structure} $U^\settt{1}$. In the present work, we will only use the following more
down-to-earth version of the scalar action:

\begin{theorem}[Rescaling]\label{th:scalars}
The group $\K^\times$ acts on $U^\sett{1}$ by automorphisms:
fix a scalar $s \in \K^\times$ and define $\Phi_s:U^\sett{1} \to U^\sett{1}$ by 
\begin{align*}
U^{[1]} \to U^{[1]}, & \quad (x,v,t) \mapsto \Phi_s (x,v,t):=(x,sv,ts^{-1}) ,
\\
U \times \K  \to  U \times \K, & \quad
(x,t) \mapsto \Phi_s(x,t) := (x,ts^{-1}) .
\end{align*}
Then  $\Phi_s$ is an automorphism of $U^\sett{1}$,
and $\Phi_{st}=\Phi_s \Phi_t$, $\Phi_1 = \id$. Moreover,
the finite part $U_\fin^\sett{1}$, and the tangent bundle $TU$, are stable under $\Phi_s$.
\end{theorem}

\begin{proof}
The action is well-defined: this follows from
 $\alpha(x,sv,ts^{-1})= (x,ts^{-1}) = s.\alpha(x,v,t)$ and
$\beta(x,sv,ts^{-1}) = (x + ts^{-1}sv,ts^{-1})=(x+tv,ts^{-1})=s.\beta(x,v,t)$.
By direct check, for each $s \in \K^\times$, the formulae from the theorem define an automorphism.
Since $ts^{-1} \in \K^\times$ if $t,s \in \K^\times$, the finite part is stable, and since
$0s^{-1}=0$, it follows that $TU$ is stable.
\end{proof}

\subsection{Tangent maps}
Every map $f$ extends to a morphism of finite parts of tangent groupoids.
By ``extends'' we mean that the base map, on the level of objects, is $f$ itself, resp.\ $f \times \id_\K$.
On the level of the total set of the groupoid, the extended map is essentially given by the difference quotient map
$f^{[1]}$ defined by (\ref{eqn:differencequotient}):
given $\K$-modules $V,V'$,  non-empty subsets $U \subset V, U' \subset V'$ and a map $f:U \to U'$,  let
\begin{align}
f^\sett{1}_{\rm fin}: U^\sett{1}_{\rm fin}  \to U^\sett{1}_{\rm fin} , & \quad
(x,v,t) \mapsto \bigl( f(x) , f^{[1]}(x,v,t) ,t \bigr)  ,
\\
f^\sett{1}_t : U^\sett{1}_t  \to U^\sett{1}_t , & \quad
(x,v) \mapsto \bigl( f(x) , f^{[1]}(x,v,t)  \bigr) ,
\end{align}
where in the second line $t \in \K^\times$ is fixed.

\begin{theorem}[Tangent maps]\label{th:tangentmap}
The map $f^\sett{1}_{\rm fin}: U^\sett{1}_{\rm fin}  \to U^\sett{1}_{\rm fin}$ is a functor, and so is
$f^\sett{1}_t : U^\sett{1}_t  \to U^\sett{1}_t$ for each fixed $t \in \K^\times$. 
The functor $f^\sett{1}_\fin$ commutes with each automorphism $\Phi_s$ with  $s \in \K^\times$:
$f^\sett{1}_\fin \circ \Phi_s = \Phi_s \circ f^\sett{1}_\fin$.
\end{theorem}

\begin{proof}
Once more, we invite the reader to check by direct computation that
 properties (1), (2) from  \ref{ssec:functors} hold (see \cite{Be15} for  detailed computations). 
E.g., 
\begin{align*}
\beta \circ f_t^\sett{1} (x,v)& = f(x) + t \, \frac{f(x+tv)-f(x)}{t} = f(x+tv) = f\circ \beta(x,v),
\end{align*}
and property (2) is directly proved from (\ref{eqn:add}).
More conceptually, these computations may be interpreted as follows: for invertible $t$, the
anchor isomorphism $\Upsilon$ from Theorem \ref{th:anchor} intertwines $f_t^\sett{1}$ and $f \times f$,
$$
\Upsilon  \circ f_t^\sett{1} (x,v) = \bigl(\alpha ( f_t^\sett{1}(x,v)), \beta (f_t^\sett{1}(x,v)) \bigr) = (f(x),f(x+tv)) =
(f \times f) \circ \Upsilon(x,v).
$$
Now,
it is easily checked that $(f \times f,f)$ is a morphism $\pG(U) \to \pG(U')$, hence, via
$\Upsilon$,  $f_t^\sett{1}$ is also groupoid morphism. 
On the level of finite parts, via $\Upsilon$, the morphism corresponds to
$(f \times f \times \id_{\K^\times}, f \times \id_{\K^\times})$. 
In the same way, $\Phi_s$ corresponds to
$(\id_U \times \id_U \times s^{-1} \id_\K, \id_U \times s^{-1} \id_\K)$,
which obviously commutes with the morphism given by the preceding formulas.
\end{proof}

\nin
A map $f$ extends to a functor of tangent groupoids if, and only if, it is $C^1$:

\begin{theorem}[Topological calculus]\label{th:topcal}
Assume that $\K$ is a good topological ring, $V,V'$ topological $\K$-modules and
$U \subset V,U \subset V'$ open, and $f:U \to U'$. Then the following are equivalent:
\begin{enumerate}
\item
$f$ is of class $C^1$ over $\K$,
\item
the finite part $f^\sett{1}_{\rm fin}$ from the preceding theorem extends to a continuous functor
 $f^\sett{1}:U^\sett{1} \to (U')^\sett{1}$. 
\end{enumerate}
If this is the case, $f^\sett{1}$ commutes with the $\K^\times$-action, as in the preceding theorem, and, 
for $t=0$, the {\em tangent map} $Tf:=f_0:TU \to TU'$ is {\em linear} in fibers:
$$
\forall x \in U, v,v' \in V, s \in \K: \qquad 
\begin{matrix}
Tf(x,v+v') &= & Tf(x,v) + Tf(x,v'), \\ Tf(x,sv) &= &s.Tf(x,v) .
\end{matrix}
$$
\end{theorem}

\begin{proof}
The proof is spelled out in full detail in \cite{Be15}:
(1) is equivalent to saying that the difference quotient map $f^{[1]}$ extends, which in turn is equivalent to saying that
$f^\sett{1}(x,v,t)=(f(x),f^{[1]}(x,v,t),t)$ extends to a continuous map on $U^{[1]}$. 
We have to prove that this extended map still is a functor commuting with the scalar action.
But this follows from the ``density principle'' (Lemma \ref{la:density}) and
 the fact that the finite part is a functor. 
(This is essentially the argument from Example \ref{ex:lin}.)
\end{proof}

\subsection{Chain rule: the ``derivation functor''}
Most of the basic results of calculus carry over to topological calculus, and the proofs are very simple: prove the
claim by direct computation for invertible scalars $t$, then by continuity and density the result carries over to $t=0$.
Here an example:

\begin{theorem}[Chain rule]
Let $U,U',U''$ be open in topological  $\K$-modules $V,V',V''$, respectively, and  $g:U' \to U''$ and $f:U \to U'$. 
Then, if $g$ and $g$ are $C^1$, then so is $g \circ f$, and we have the chain rule
$$
(g \circ f)^\sett{1}=g^\sett{1} \circ f^\sett{1} ,
$$
or, equivalently, $\forall t \in \K$:
$(g \circ f)_t^\sett{1}=g_t^\sett{1} \circ f_t^\sett{1}$. In particular,
$T(g\circ f)=Tf \circ Tg$.
\end{theorem}

\begin{proof}
A proof by direct computation is given in \cite{Be15}. In a conceptual way, that proof may be presented as follows:
for $t \in \K^\times$, as in the proof of th.\ \ref{th:tangentmap}, via the anchor
 isomorphism $\Upsilon$,  the chain rule translates to
$(g \circ f )  \times (g \circ f) = (g \times g) \circ (f \times f)$, which clearly is true. 
By the Density Lemma \ref{la:density}, equality holds for all $t \in \K$, and hence in particular for $t=0$, whence the usual
chain rule.
\end{proof}

\nin
The ``derivation symbol''  $\{ 1 \}$ is thus a functor from the  category of (open) subsets of topological $\K$-modules, with 
$C^1$-maps as morphisms, to the category of (topological) groupoids with their (continuous) morphisms.
Topological differential calculus is the theory of this functor. 
Of course, now we must talk about {\em second and higher order calculus}: what happens if we apply this functor {\em several times}?
The first thing we have to do is to ``copy and save''   our functor:

\begin{definition}\label{def:copies}
For every $n \in \N$, we denote by $\{ n \}$, $\{ n \}_t$, $U^\sett{n}$, $f^\sett{n}$, etc., a copy, called {\em of $n$-th
generation},  of the objects defined above for $n=1$.
\end{definition}

\nin
Before explaining what to do with these copies, let's pause for a more classical intermezzo:

\section{Intermezzo on manifolds}\label{sec:manifolds} 

\subsection{Manifolds}
By general principles, the derivation functor $\sett{1}$ extends to the category of smooth manifolds and smooth maps:

\begin{theorem}\label{th:mfd}
For every Hausdorff manifold $M$, there is a groupoid $M^\sett{1} = (M^{[1]},M \times \K)$, agreeing with the groupoid
$U^\sett{1}$ from Theorem \ref{th:tangentgroupoid} when $M=U$ is open in a topological $\K$-module.
Smooth maps between manifolds correspond precisely to continuous functors between these groupoids.
For any fixed $t\in \K$, the groupoid $M^\sett{1}$ gives rise to a  groupoid $M^\sett{1}_t = (M_t,M)$ which is
isomorphic to $\pG(M)$ for $t \in \K^\times$, and to the tangent bundle $TM$ for $t=0$. 
There is a canonical $\K^\times$-action on $M^\sett{1}$, commuting with all functors $f^\sett{1}$.
\end{theorem}

The proof (\cite{Be15}) is quite straightforward, but in order to spell it out properly, we have to give a 
formal and precise definition of what we mean by ``manifold over general base fields or rings'': charts, atlasses, and
all that.
This is carried out in \cite{Be16}: it turns out that, formally, {\em a manifold structure (an atlas) is an ordered groupoid}.
For the purposes of the present work, it is not really necessary to go into the details; let us just mention that
the partial order structure comes from the natural inclusion of charts, and the groupoid structure reflects 
equivalence of charts if they have same chart domain. 
Using this language, we can describe the local procedure of
gluing together the sets $U_i^\sett{1}$ from chart domains $U_i$, using the chain rule, to a set $M^\sett{1}$.
In the same way, the groupoid law on $M^\sett{1}$ is defined locally, near the unit section.
However,  in order to define it
globally, we need the Hausdorff assumption from the theorem (cf.\ Lemma D.3 of \cite{Be15}:  to define 
$a \ast b$, if $a,b$ are sufficiently close to each other, we can work in one connected local chart, but else
we have to use possibly non-connected chart domains obtained from two disjoint chart domains which exist
due to the Hausdorff assumption. Without that assumption we would only get {\em local groupoids}, which suffices
for many purposes.
If $\K$ is a field, the gluing procedure can be avoided by presenting the tangent groupoid ``\`a la Connes'' (cf.\
def.\ \ref{def:tangentgroupoid} and footnote there), and thus 
 this item seems not to be related to questions involving non-Hausdorff groupoids studied, e.g., in
Non-commutative Geometry.)

\subsection{Lie groups and Lie groupoids}

\begin{definition}
A {\em Lie group} is a group $(G,e,\cdot)$ together with a manifold structure such that the group law
$\cdot$ and inversion are differentiable.
A {\em Lie groupoid} is a groupoid $G = (G_1,G_0,\alpha,\beta,1,\ast,i)$ together with manifold structures on
$G_1,G_0$ and on $G_1 \times_{\alpha,\beta} G_0$ such that all structure maps $\alpha,\beta,1,\ast,i$ are 
differentiable.\footnote{We follow here the pattern of the general definition given in the $n$-lab,
\url{https://ncatlab.org/nlab/show/Lie+groupoid}. 
Of course, under suitable assumptions some conditions may be weakened, e.g., 
in \cite{Ma05}, def.\ 1.1.3, it is required that $\alpha,\beta$ be submersions, which in the real finite dimensional
case implies that 
that $G_1 \times_{\alpha,\beta} G_0$ is a manifold. 
In our setting, this implication does in general not hold.}
\end{definition}

\begin{theorem}\label{th:Liegroupoid1}
Let $U$ open in $V$ and $t\in \K$. Then $U^\sett{1}$ and $U^\sett{1}_t$ are Lie groupoids. 
Likewise, if $M$ is a Hausdorff manifold,  $M^\sett{1}$ and $M^\sett{1}_t$ are Lie groupoids.
\end{theorem}

\begin{proof}
Since $U$ is open in $V$, $U^{[1]} = \{ (x,v,t) \mid x+tv \in U \}$ is open in $V \times V \times \K$, and the
set $U^{[1]} \times_{\alpha,\beta} U^{[1]}= \{ (x',v',t';x,v,t) \mid t=t', x'=x+tv \in U ,x'+tv' \in U \}$ is naturally identified with
$$
\{ (x,v,v',t) \mid x \in U, \, v,v' \in V, \, t \in \K :  x+tv \in U, \, x+t(v+v') \in U \}
$$
which is open in $V^3 \times \K$. Thus these three sets are smooth manifolds (with atlas a single chart induced by the
ambiant linear space), and  all structure maps are smooth since they are all given by explicit formulas
involving only scalar multiplication and vector addition, which are continuous, whence differentiable. 
Again, by the principles explained above, the result carries over to the manifold level.
\end{proof}

\nin
What we have seen so far implies that  a Lie group, or a Lie groupoid, carries $3$ groupoid structures, that are compatible
with each other:
first, it is a group (resp.\ groupoid) in its own right;
second, as said above, its manifold structure is an (ordered) groupoid;
third, by Theorem \ref{th:mfd}, $G^\sett{1}$ carries the tangent groupoid structure.
It is time to explain what it means to say that ``one groupoid structure is compatible with another''.
Even if we neglect the ordered groupoid structure corresponding  to the atlas, there remains a
{\em double groupoid structure}.
And we have not even started to develop {\em higher order calculus}, where similar considerations
lead to  {\em $n$-fold groupoids}.

\section{Double and higher groupoids}\label{sec:doublegroupoids} 

Higher order calculus arises by {\em iterating} the operation of ``differentiation'', giving rise to things like
$f'', f'''$, or $\partial_u \partial_v f$, or
 $d(df)$, or $T(Tf)$... Such iteration procedures may look harmless, but can lead to complicated objects.
 For instance, let's compute  the {\em second order slope} $f^{[2]} =(f^{[1]})^{[1]}$: it is given by
$ f^{[2]} \bigl( ( v_0,v_1,t_1), (v_2,v_{12},t_{12}) , t_2 \bigr) = $ 
\begin{align}\label{eqn:4.1}
\quad \qquad  &=  \frac{1}{t_2} \Bigl(    f^{[1]} \bigl((v_0,v_1,t_1) + t_2 
(v_2,v_{12},t_{12})\bigr) - f^{[1]}(v_0,v_1,t_1)  \Bigr)
\cr
\quad \qquad  & =
\frac{
f\bigl(v_0 + t_2 v_2 + (t_1+ t_2 t_{12}) (v_1+ t_2 v_{12}) \bigr) - f(v_0 + t_2 v_2)}{t_2(t_1 + t_2 t_{12}) }
-
\frac{f(v_0+t_1 v_1) - f(v_0)}{t_2 t_1}
\end{align}
and it extends, if $f$ is $C^2$, to a map $f^{[2]}$ defined on
  the set $U^{[2]} = (U^{[1]})^{[1]}$  given by 
\begin{equation*}
\Bigsetof{ (v_0,v_1,v_2,v_{12},t_1,t_2,t_{12}) \in V^4 \times \K^3 } {\begin{array}{c}
v_0 \in U \\  v_0 + t_1 v_1 \in U,  \\
  v_0 + t_2 v_2 \in U \\ v_0 + t_2 v_2 +  (t_1+t_2 t_{12} )(v_1+  t_2 v_{12}) \in U \end{array}} .
\end{equation*} 
Clearly, it is hopeless to try to understand $f^{[n]}$ for $n\geq 3$ by writing out an ``explicit formula'' like
(\ref{eqn:4.1}) -- we need a more conceptual approach. The notion of {\em $n$-fold groupoid} provides such a
conceptual framework. In the setting described above,
 we apply the ``derivation symbol'' $\sett{ 1}$ several times:
 first, it gives a groupoid $U^\sett{1}$, and next a {\em double groupoid} $(U^\sett{1})^\sett{1}$, and so on.
Moreover,  we shall see that the outcome of this iteration
depends on the order in which things are performed, hence our {\em notation} has to take account of that:
 we will apply first the operator
$\sett{1}$, then its copy $\sett{2}$, and write
$U^\sett{1,2}:=(U^\sett{1})^\sett{2}$, and so on (see eqn. (\ref{3})).

\subsection{Ehresmann's definition}
Following Charles Ehresmann, one can define double and higher groupoids in a very short way (reproduced, e.g., on
the $n$-lab):

\begin{definition} A $0$-fold groupoid is just a set. 
A {\em (strict) $n$-fold groupoid} is a groupoid internal to the category of (strict)  $(n-1)$-fold groupoids.
\end{definition}

\nin
The drawback of this short definition is that it is not very explicit, and moreover that it uses the vocabulary of ``big'' categories
in order to define something ``small'', that is, an object of usual algebra. Let us give definitions avoiding these drawbacks.
Since all our structures will be ``strict'', we suppress this term in the sequel. First of all, we spell out Ehresmann's definition in
more detail:

\begin{definition}
An {\em $n$-fold groupoid} for $n=0$ is just a set without structure,  morphisms being ordinary maps,
 and for $n=1$, it is a pair of sets
$G=(G_0,G_1)$ with structure maps $\alpha,\beta,1,i,\ast$ as in Def.\ \ref{def:groupoid}, and morphisms are functors $f=(f_0,f_1)$ as
defined in \ref{ssec:functors}. 
For $n\geq 1$, it is a groupoid $G = (G_0,G_1,\alpha,\beta,1,i,\ast)$, such that:
\begin{enumerate}
\item
$G_0$ and $G_1$ carry each the structure of an $(n-1)$-fold groupoid,
\item
$G_1 \times_{\alpha,\beta} G_1$ is a sub-$(n-1)$-fold groupoid of $G_1 \times G_1$,
\item
the structure maps $\alpha,\beta,1,i,\ast$ are morphisms of $(n-1)$-fold groupoids.
\end{enumerate}
A {\em morphism of $n$-fold groupoids} is a groupoid morphism $f=(f_0,f_1)$ such that both
$f_0$ and $f_1$ are morphisms of $(n-1)$-fold groupoids.
\end{definition} 

 \subsection{The Brown-Spencer definition of double groupoids}
 In \cite{BrSp76}, Brown and Spencer give a ``purely algebraic'' definition of double groupoids, in terms of 
 structure maps and defining algebraic identities. This is obtained by writing out, for
  $n=2$, the preceding definition in full detail:
$G_1 = (G_{11},G_{10})$ and $G_0 = (G_{01},G_{00})$ are groupoids,
$\alpha = (\alpha_1:G_{11} \to G_{01},\alpha_0:G_{01} \to G_{00})$, and likewise $\beta$, are groupoid morphisms,
and so are the unit sections; that is, we have $4$ sets and diagrams of mappings between them:
\begin{equation}\label{eqn:doublediagram}
\begin{matrix}
G_{11}  &    \rightrightarrows & G_{01}  \cr
 \downdownarrows & &  \downdownarrows 
 \cr
G_{10} &     \rightrightarrows & G_{00}
 \end{matrix} , \qquad \qquad
\begin{matrix}
G_{11}  &   \stackrel { }  \leftarrow & G_{01}  \cr
 \uparrow & & \uparrow 
 \cr
G_{10} &   \stackrel { }  \leftarrow & G_{00} 
 \end{matrix}  ,
\end{equation}
as well as products $\ast$ on $G_{11}$ and $G_{01}$ and $\bullet$ on $G_{11}$ and $G_{10}$,
such that 
\begin{enumerate}
\item
each of the four edges of these diagrams with its structure maps is a groupoid,
\item
each pair of corresponding projections (like $(\alpha_{1}:G_{11}\to G_{10},\alpha_0:G_{01}\to G_{00})$) 
and each pair of unit sections is a morphism of groupoids,
\item
the product $\ast$ is a morphism from $(G_{11} \times_{G_{10}} G_{11},\bullet \times \bullet)$ to $(G_{01},\bullet)$
(and likewise for $\bullet$ and $\ast$ exchaged).
\end{enumerate}

\nin
Whereas it is straightforward to write (1) and  (2)  in equational form (like, e.g.,
$\alpha_1 (b \ast a ) =\alpha_1(b) \ast \alpha_1(a)$, cf.\ \cite{Be15}), this is slightly less obvious for (3):
 the map $A:=\ast :G_{11} \times_{G_{10}} G_{11} \to G_{11}$, $(a,b) \mapsto a \ast b$ is  a morphism for $\bullet$ iff
$$
A  ((a,b) \bullet  (c,d)) = A (a,b) \bullet A(c,d) ,
$$
that is, iff the  following {\em interchange law} holds:
\begin{equation}\label{eqn:interchange}
(a \bullet c) \ast  (b \bullet d) = (a \ast b) \bullet (c \ast d) ;
\end{equation}
Summing up, a double groupoid is given by four sets $(G_{11},G_{10},G_{01},G_{00})$ and certain structure maps
satisfying algebraic conditions expressing (1) -- (3), like (\ref{eqn:interchange}).
We shall often indicate double groupoids by diagrams of the form (\ref{eqn:doublediagram}).

\begin{remark}
It follows from (1), (2), (3) that inversion of $\ast$ is an automorphism of $\bullet$ -- which may look surprising since
it is an {\em anti}automorphism for $\ast$. So, in the particular case where $\ast = \bullet$, both must be commutative (cf.\
example \ref{ex:doublegroup} below).
\end{remark}

\begin{example}[The pair groupoid of a groupoid]\label{ex:PGL} 
Let $L=(L_1,L_0)$ be a groupoid.
Then the pair groupoid $\pG(L)$ of $L$ is a double groupoid:  
$$
\begin{matrix}
L_1 \times L_1  &   \rightrightarrows & L_{1}  \cr
 \downdownarrows  & &   \downdownarrows 
 \cr
L_{0} \times L_0 &    \rightrightarrows & L_{0}
 \end{matrix} 
 $$
 The horizontal groupoid laws are pair groupoids of $L_1$, resp.\ $L_0$, and the vertical ones come from the given
 one on $L$. A conceptual explanation is given by the fact that the symbol $\pG$ is a {\em product preserving 
 functor}, taking values in groupoids (cf.\ next chapter). 
In particular, taking $L = \pG(M)$, the pair groupoid of a set $M$, we get the {\em double pair groupoid}
$\pG^\sftwo(M)= \pG (\pG(M))$ of $M$:
$$
\begin{matrix}
M^4  &   \rightrightarrows & M^2  \cr
 \downdownarrows  & &   \downdownarrows 
 \cr
M^2  &    \rightrightarrows & M
 \end{matrix} 
 $$
\end{example}

\begin{example}[Double groups]\label{ex:doublegroup}  A {\em double group} is  a double groupoid of the form 
$$
\begin{matrix}
G_{11}  &    \rightrightarrows & 1   \cr
  \downdownarrows  & &   \downdownarrows 
 \cr
1  &    \rightrightarrows & 1 
 \end{matrix} 
$$
that is, a set $G=G_{11}$ with a single unit $1$  and two group laws $\ast$ and $\bullet$ satisfying the interchange law.
We infer
$a\ast b = (a \bullet 1) \ast (1 \bullet b)= (a \ast 1) \bullet (1 \ast b)= a \bullet b$, whence $\ast = \bullet$, and now the interchange
law implies that the group must be {\em commutative}.
Conversely, every commutative group does indeed define a double group. 
This apparenty trivial observation  explains why {\em abelian} groups lie at the bottom of so many
mathematical structures: they ``are'' precisely the double groups.
\end{example}

\subsection{Notation, hypercubes, and small characterization}
It should be obvious now that a $3$-fold groupoid will consist of $8$ sets, each corresponding to the vertex of a cube,
 and so on:
an $n$-fold groupoid is given by $2^n$ sets that correspond to the vertices of an {\em $n$-hypercube}. It is now time to
improve our  {\em notation}: 

\begin{definition}
Let $N \subset \N$ be a finite subset, for instance, the {\em standard subset} $\sfn$ given by (\ref{eqn:sfn}). 
The {\em $N$-hypercube} has {\em vertex set} $\cP(N)$ (power set of $N$), and {\em edges} $(B,A)$, where 
$B \subset A \subset N$, and
$A$ has one element more than $B$.
We denote such an edge by $\ol{BA}$.
 A {\em face} is given by four vertices
$(D,C,B,A)$ such that $\ol{DC},\ol{DB},\ol{BA},\ol{CA}$
are edges.
\end{definition}

\begin{theorem}[Small characterization of $n$-fold groupoids]\label{th:smallchar}
An $n$-fold groupoid is given by $2^n$ sets $(G^A)_{A \in \cP(\sfn)}$, indexed by the natural hypercube $\cP(\sfn)$, 
and structure maps, satisfying:
\begin{enumerate}
\item
for each {\em edge} $(B,A)$, we have projections $\alpha^{A,B},\beta^{A,B}$, unit sections $1^{A,B}$,  inversions
$i^{A,B}$ and products $\ast^{A,B}$ turning  $(G^A,G^B)$ into a {\em groupoid},
\item
for each {\em face} $(D,C,B,A)$ we have a {\em double groupoid} (as defined algebraically in the preceding subsection)
$$
\begin{matrix}
G^A  &    \rightrightarrows & G^C  \cr
\downdownarrows  & & \downdownarrows 
 \cr
G^B  &    \rightrightarrows & G^D .
 \end{matrix} 
$$
\end{enumerate}
\end{theorem}

\begin{remark}
{\em Small $n$-fold categories} are defined and characterized in the same way, just by forgetting the inversion maps.
\end{remark}

\begin{remark}
Here, the {\em total set} of the hypercube is $\sfn$. But one may define in the same way $n$-fold groupoids with any total
set $N\subset \N$ such that $\vert N \vert = n$, and then use the {\em notation}
$G^{A;N}$ for the {\em vertex sets} and $\alpha^{B,A;N}$ etc. for the {\em edge projections}.
\end{remark}

\nin
The proof  of the theorem, by induction, is straightforward (see \cite{Be15b}, Th.\ B.2).
To illustrate, say, the induction step from $n=3$ to $n+1=4$,  consider Figure 3 showing a 
 {\em tesseract} ($4$-cube).
In the figure, vertices are labelled by $ij$, to abbreviate  $\{i,j\}$, etc. 
Let us call a vertex $A$
\begin{enumerate}
\item[\emph{old}] 
if $n+1 \notin A$,
\item[\emph{new}]
if $n+1 \in A$; then $A = B \cup \{ n+1 \}$, where $B$ is an ``old'' vertex.
\end{enumerate}
The old vertices form a $3$-cube (on the left), and so do the new vertices (right). 
Now,
the proof of the theorem consists, essentially, in contemplating this figure.
The result is likely to be folklore among specialists
in higher category theory. However, \cite{FP10}  is the only reference I was able to find. 
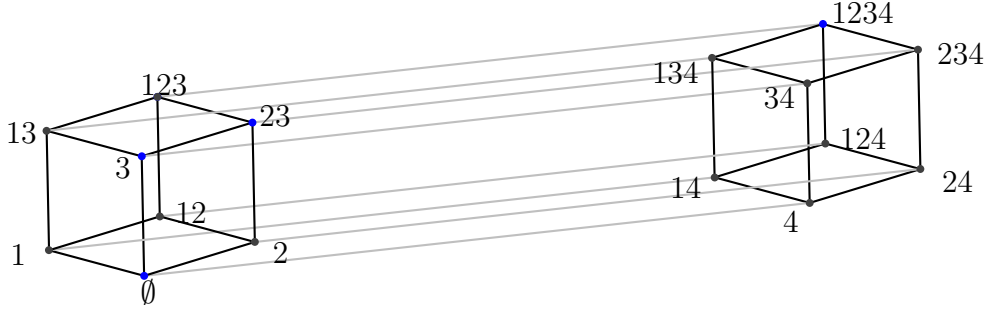
\begin{figure}[h]
\caption{A tesseract by assembling two cubes}
\begin{center}
\psset{xunit=0.8cm,yunit=0.5cm,algebraic=true,dotstyle=o,dotsize=3pt 0,linewidth=0.8pt,arrowsize=3pt 2,arrowinset=0.25}
\begin{pspicture*}(-0.22,-3.99)(19.84,4.06)
\psline[linecolor=lightgray](14.42,3.48)(3.36,1.54)
\psline(3.36,1.54)(4.94,0.86)
\psline(4.94,0.86)(3.1,-0.04)
\psline(3.1,-0.04)(3.14,-3.22)
\psline(3.36,1.54)(1.52,0.64)
\psline(1.52,0.64)(3.1,-0.04)
\psline[linecolor=lightgray](1.52,0.64)(12.58,2.58)
\psline(14.42,3.48)(12.58,2.58)
\psline(1.52,0.64)(1.56,-2.54)
\psline(1.56,-2.54)(3.14,-3.22)
\psline(14.42,3.48)(16,2.8)
\psline(16,2.8)(14.16,1.9)
\psline[linecolor=lightgray](3.1,-0.04)(14.16,1.9)
\psline(12.58,2.58)(14.16,1.9)
\psline[linecolor=lightgray](16,2.8)(4.94,0.86)
\psline(12.58,2.58)(12.62,-0.6)
\psline(14.16,1.9)(14.2,-1.28)
\psline(12.62,-0.6)(14.2,-1.28)
\psline[linecolor=lightgray](12.62,-0.6)(1.56,-2.54)
\psline[linecolor=lightgray](3.14,-3.22)(14.2,-1.28)
\psline(14.42,3.48)(14.46,0.3)
\psline(14.46,0.3)(12.62,-0.6)
\psline(16,2.8)(16.04,-0.38)
\psline(16.04,-0.38)(14.2,-1.28)
\psline(16.04,-0.38)(14.46,0.3)
\psline[linecolor=lightgray](3.4,-1.64)(14.46,0.3)
\psline(3.4,-1.64)(4.98,-2.32)
\psline[linecolor=lightgray](4.98,-2.32)(16.04,-0.38)
\psline(4.98,-2.32)(3.14,-3.22)
\psline(3.4,-1.64)(1.56,-2.54)
\psline(4.94,0.86)(4.98,-2.32)
\psline(3.36,1.54)(3.4,-1.64)
\rput[tl](3.08,-3.32){$ \emptyset $}
\rput[tl](0.92,-2.4){$ 1 $}
\rput[tl](5.28,-2.34){2}
\rput[tl](2.66,-0.08){3}
\rput[tl](0.84,0.86){13}
\rput[tl](5.06,1.3){23}
\rput[tl](3.08,2.14){123}
\rput[tl](13.76,-1.5){4}
\rput[tl](16.4,-0.44){24}
\rput[tl](16.32,2.96){234}
\rput[tl](14.56,4.06){1234}
\rput[tl](11.88,-0.62){14}
\rput[tl](3.66,-1.26){12}
\rput[tl](14.7,0.7){124}
\rput[tl](11.58,2.46){134}
\rput[tl](13.44,1.78){34}
\begin{scriptsize}
\psdots[dotstyle=*,linecolor=blue](14.42,3.48)
\psdots[dotstyle=*,linecolor=blue](3.36,1.54)
\psdots[dotstyle=*,linecolor=blue](4.94,0.86)
\psdots[dotstyle=*,linecolor=blue](3.1,-0.04)
\psdots[dotstyle=*,linecolor=blue](3.14,-3.22)
\psdots[dotstyle=*,linecolor=darkgray](1.52,0.64)
\psdots[dotstyle=*,linecolor=darkgray](12.58,2.58)
\psdots[dotstyle=*,linecolor=darkgray](1.56,-2.54)
\psdots[dotstyle=*,linecolor=darkgray](14.16,1.9)
\psdots[dotstyle=*,linecolor=darkgray](3.36,1.54)
\psdots[dotstyle=*,linecolor=darkgray](16,2.8)
\psdots[dotstyle=*,linecolor=darkgray](12.62,-0.6)
\psdots[dotstyle=*,linecolor=darkgray](14.2,-1.28)
\psdots[dotstyle=*,linecolor=darkgray](14.46,0.3)
\psdots[dotstyle=*,linecolor=darkgray](16.04,-0.38)
\psdots[dotstyle=*,linecolor=darkgray](3.4,-1.64)
\psdots[dotstyle=*,linecolor=darkgray](4.98,-2.32)
\end{scriptsize}
\end{pspicture*}
\end{center}
 \end{figure}

\section{First order Lie calculus}\label{sec:firstLie}

\subsection{General principles}
The approach to Lie theory pursued in \cite{Be08}, strongly motivated by the theory of {\em product preserving
functors} from \cite{KMS93}, starts by the classical remark that, if $(L,m,i,1)$ is a Lie group, then so is its
tangent bundle $(TL,Tm,Ti,T1)$, with group laws the tangent maps of the group laws $m,i$ of  $L$ and unit
$T1=0_1$, the zero vector in the tangent space $T_1 L$. More generally:

\begin{lemma}\label{la:principle} 
Assume $F$ is a {\em product preserving functor}, i.e., a functor commuting with cartesian products
in the sense that always $F (A \times B) = F(A) \times F(B)$. 
Then, if $(G,m,1)$ is a group, so is $(FG,Fm,F1)$, and if
$(\K,a,m,0,1)$ is a unital ring (with addition map $a$ and multiplication map $m$), then so is
$(F\K,Fa,Fm,F0,F1)$.
\end{lemma}

\begin{proof}
Write the defining properties of a group, resp.\ of a ring, as commutative diagrams, involving structure maps,
cartesian products and diagonal imbeddings. Applying $F$ to such a diagram yields a diagram of the same form,
and hence a structure of the same kind. (Cf.\ \cite{Be08, KMS93} for explicit forms of such diagrams and for 
more examples of such functors, besides the tangent functor $T$.)
\end{proof}

\subsection{From groupoids to double groupoids}
The preceding lemma also applies to groupoids, taking for $F$ a functor  $\{ 1 \}_t$ which is product preserving. 
Now, the new feature is  that each functor $\sett{1}_t$ takes itself values in groupoids (and not only in sets without 
specified structure), which implies that $\sett{1}_t$, applied to a groupoid, gives us a {\em double groupoid}:

\begin{theorem}\label{th:double}
Let $L = (L_1,L_0)$ be a Lie groupoid.
Then, applying the derivation symbol $\sett{1}$, 
resp.\ $\sett{1}_t$ for fixed $t \in \K$,
we get a double groupoid
$$
\begin{matrix}
L_1^\sett{1}  &   \rightrightarrows & L_0^\sett{1}  \cr
 \downdownarrows & &  \downdownarrows 
 \cr
L_1 \times \K  &    \rightrightarrows & L_0 \times \K ,
 \end{matrix} 
\qquad \mbox{ resp. }  \qquad
\begin{matrix}
(L_1)_t   &   \rightrightarrows & (L_0)_t   \cr
 \downdownarrows  & &  \downdownarrows 
 \cr
L_1    &    \rightrightarrows & L_0 .
 \end{matrix} 
$$
\end{theorem}

\begin{proof}
In both diagrams, the vertical double arrows stand for the groupoid structures  given by Theorem
\ref{th:tangentgroupoid} (let us denote by $\bullet$ its groupoid product), 
and the upper level horizontal double arrows come from applying our functor $\sett{1}$, resp.\
$\sett{1}_t$, to the structure maps of $L$ appearing in the corresponding place of the lower level horizontal arrows. 
According to Theorem \ref{th:tangentmap}, such horizontal pairs are morphisms of the vertical groupoids. 
The lower horizontal edges are groupoids since $L$ is, by assumption, a groupoid.
Let us prove that the upper horizontal edges also describe groupoids:
as explained in Lemma \ref{la:principle}, for each fixed $t \in \K$, it suffices to show that $\{ 1 \}_t$ is a
 {\em product preserving functor}: indeed,
\begin{align*}
(U \times U')_t & = \big\{ (x,x',v,v') \in (U \times U') \times (V \times V') \mid \, (x,x')+t(v,v') \in U \times U'  \big\}
\\
& = \big\{ (x,x',v,v') \in U \times U' \times V \times V'  \mid \,  x+tv \in U, x'+tv' \in U'  \big\}
\\
& \cong  \big\{ (x,v) \in U \times V  \mid \,  x+tv \in U  \big\} \times \big\{ (x',v') \in U' \times V ' \mid \,  x'+tv' \in U ' \big\}
\\
&  = U_t \times U_t'
\end{align*}
Thus, by the lemma, on the top line we have a groupoid with product $\ast^\sett{1}$, source projection $\alpha^\sett{1}$, etc.  
Moreover, for any map $f$,
the vertical projections intertwine $f^\sett{1}$ and $f\times \id_\K$, 
which means that vertical pairs of projections are groupoid morphisms.
Finally, taking $\ast$ for $f$, from $f^\sett{1} (a \bullet b) = f^\sett{1}(a) \bullet f^\sett{1}(b)$, we get 
that $\ast^\sett{1}$ is a morphism for $\bullet$, i.e., the interchange law holds. 
\end{proof}

\begin{remark}
Please note that the functor $\sett{1}_t$ is product preserving only for fixed $t$ (which is all we need to prove the
preceding theorem). The functor $\sett{1}$ is {\em not} product preserving, but satisfies the rule
$(A \times_C B)^\sett{1} = A^\sett{1} \times_{C^\sett{1}} B^\sett{1}$, which is the good one to generalize Lemma 
\ref{la:principle} to groupoids
 (cf.\ \cite{Be15}).
\end{remark}

\begin{remark}
When $t$ is invertible, Theorem  \ref{th:anchor} implies that $L_t$ is isomorphic to the double groupoid $\pG(L)$
(see Example \ref{ex:PGL}).
\end{remark}

\begin{example}
If $L$ is a Lie group, that is, $L_0 = 1$, $L_1 = L$, we get double groupoids
$$
\begin{matrix}
L^\sett{1}  &   \rightarrow &  \K   \cr
 \downdownarrows & &  {}_\id   \downarrow  \phantom{\pi} 
 \cr
L  \times \K  &  \rightarrow &  \K , 
 \end{matrix} 
\qquad \mbox{ resp. } \qquad
\begin{matrix}
L_t   &   \rightarrow &  1   \cr
 \downdownarrows & &  {}_\id \downarrow  \phantom{\pi} 
 \cr
L    &     \rightarrow &  1 . 
 \end{matrix} 
$$
Indeed,
this is a degenerate case:  $L_0 = 1$, and $1^\sett{1} = \K$ is a trivial groupoid.
\end{example}

\subsection{From $n$-fold groupoids to $(n+1)$-fold groupoids} By the same principles:

\begin{definition}
An {\em  $n$-fold Lie groupoid} is an $n$-fold groupoid $(L^A)_{A\in \cP(\sfn)}$ such that, for each
edge $(B,A)$ of the natural hypercube, the edge groupoid $(L^A , L^B)$
 carries a structure of Lie groupoid.
\end{definition}

\begin{theorem}\label{th:doubledouble}
Assume $L = (L^A)_{A\in \cP(\sfn)}$ is an $n$-fold Lie groupoid. 
Then, applying the derivation symbol $\sett{n+1}$, 
resp.\ $\sett{n+1}_t$ for fixed $t \in \K$,
we get an $(n+1)$-fold groupoid $G=(G^A)_{A \in \cP(\mathsf{n+1})}$
given by the families of vertex sets:
$$
G^A = \Big\{ \begin{matrix} (L^A)^\sett{n+1} & \mbox{ if }  A \subset \sfn, \\ 
L^B \times \K & \mbox{ if }   A = B \cup \{ n+1 \} ,
\end{matrix}
\, \mbox{ resp. } \, 
G^A = \Big\{ \begin{matrix} (L^A)_t  & \mbox{ if }  A \subset \sfn, \\ 
L^B  & \mbox{ if }   A = B \cup \{ n+1 \} .
\end{matrix}
$$
\end{theorem}

\begin{proof}
One uses language from the proof of Theorem  \ref{th:smallchar} and
arguments as in the proof of Theorem \ \ref{th:double}:
the ``old'' vertices and their edges form an $n$-fold groupoid, a copy of the one we started with, $L$. 
The ``new'' vertices and their edges form another $n$-fold groupoid, obtained from the old one by applying the
functor $\sett{n+1}$, resp.\ the product-preserving functor $\sett{n+1}_t$.
Each edge joining an old vertex $B$ and a new vertex $A = B \cup \{ n+1\}$ defines a groupoid of the form
given by th.\ \ref{th:tangentgroupoid}. Each face defines a double groupoid, by the arguments given in the proof of
Theorem  \ref{th:double}.
\end{proof}

\begin{definition}
The $(n+1)$-fold groupoid $G$ obtained from an $n$-fold Lie groupoid $L$ as in the theorem, will be called the
{\em derived higher groupoid} and denoted by
$L^\sett{n+1}$, resp.\ by $L_t^\sett{n+1}$.
\end{definition}

\begin{remark}[Why the order matters]
In the same way, we could ``derive'' an $n$-fold Lie groupoid
$L = (L^A)_{A\in \cP(N)}$ with $N \subset \N$, to get an $(n+1)$-fold Lie groupoid 
$G= (G^A)_{A\in \cP(N')}$, where $N' = N \cup \{ k \}$ with
$k > j$ for all $j \in N$. (Without this last condition the procedure would depend on the choice of $k$ in an essential
way, and hence would not be well-defined!)
\end{remark}

\section{Higher order calculus}\label{sec:higherorder}

Now we are ready to iterate $n$-times the two functors $\sett{1}$ and $\sett{1}_t$ (for fixed $t$) from first order calculus. 
Both iterations give us, by the general principles developed so far, $n$-fold groupoids, denoted by $M^\sfn$ (``first construction'': full cubic),
resp.\  $M^\sfn_\ttt$ for $\ttt \in \K^n$ fixed (``second construction'': symmetric cubic).
Although the general principles are the same for both constructions, it turns out that understanding the structure of the full cubic $M^\sfn$
is far more difficult than understanding the structure of the symmetric cubic $M^\sfn_\ttt$. 
In the latter case, $M^\sfn_\ttt$ can be understood as {\em scalar extension of $M$} from $\K$ to the ring
$\K^\sfn_\ttt$, whose structure is fairly transparent, and quite close to the {\em higher order tangent rings} $T^n \K$ used in
\cite{Be08}.

\subsection{Full cubic {\sl versus} symmetric cubic}\label{ssec:full-vs-symmetric}
Recall from def.\ \ref{def:CnK} the setting of topological calculus, the
 definition of the class $C^n_\K$ and of the 
higher order slopes $f^{[n]}$ defined on the domain $U^{[n]}$. 
 Note that, if $U$ is open in $V$, then $U^{[1]}$ is open in $V^2 \times \K$, whence by induction, $U^{[n]}$ is open in
$V^{2^n}\times \K^{2^n - 1}$.
More conceptually, this kind of definition gives us
 the double groupoids $U^\sett{1,2}=(U^\sett{1})^\sett{2}$, etc.\ (recall notation from Def.\
\ref{def:copies}).
The following result is purely algebraic; no topology is used:

\begin{theorem}\label{th:Usfn}
Assume $U$ is a non-empty subset of the $\K$-module $V$.
\begin{enumerate}
\item By induction, the following defines $n$-fold groupoids:
\begin{align*}
U^\sfn & = U^\sett{1,\ldots,n}:= ((U^\sett{1})^\sett{2}\ldots )^\sett{n} , \\
U^\sfn_\fin & = U^\sett{1,\ldots,n}_\fin:= ((U^\sett{1}_\fin)^\sett{2}_\fin\ldots )^\sett{n}_\fin , 
\end{align*}
\item
for each $\ttt = (t_1,\ldots,t_n) \in \K^n$, the following defines an $n$-fold groupoid:
$$
U^\sfn_\ttt := (U^\sett{1}_{t_1})^\sett{2}_{t_2} \ldots )_{t_n}^\sett{n} .
$$
\end{enumerate}
The top vertex set of $U^\sfn$  agrees with the $n$-th order extended domain $U^{[n]}$:
\begin{equation*}
U^{\sfn;\sfn} = U^{[n]} .
\end{equation*}
Every map $f:U \to U'$ induces morphisms of $n$-fold groupoids
\begin{align*}
f^\sfn_\fin & := ((f^\sett{1}_\fin)^\sett{2}_\fin\ldots )^\sett{n}_\fin : 
U^\sfn_\fin \to (U') ^\sfn_\fin  , \\
f_\ttt^\sfn & := ((f_{t_1})_{t_2} \ldots )_{t_n} : U_\ttt^\sfn \to (U')_\ttt^\sfn  ,
\end{align*}
the latter under the condition that $\forall i = 1,\ldots,n$: $t_i \in \K^\times$.
\end{theorem}

\begin{proof}
Proceeding by induction, one uses exactly the same arguments as in the proof of theorems
\ref{th:double} and \ref{th:doubledouble}. 
To describe the top vertex set by induction, note that
$U^\sett{1}=(U^{[1]},U \times \K)$ has $U^{[1]}$ as top vertex set, so
$U^\sett{2}$ has $(U^{[1]})^{[1]}=U^{[2]}$ as top vertex set, and so on.
(Recall that the explicit formulae for these things may be quite complicated: {\it cf.}  eqn.\  (\ref{eqn:4.1}).)
\end{proof}

\begin{theorem}[Full cubic $C^n$]\label{th:tangent-n}
Let $\K$ be a good topological ring, $V,W$ topological $\K$-modules, $U \subset V$ open and
$f:U \to W$ a map. Then the following are equivalent:
\begin{enumerate}
\item
$f$ is of class $C^n_\K$,
\item
the morphism $f^\sfn_\fin$ extends to a continuous morphism
$f^\sfn: U^\sfn \to W^\sfn$.
\end{enumerate}
For every Hausdorff manifold $M$ of class $C^n_\K$, there is an $n$-fold groupoid $M^\sfn$
 such that,
when $M=U$ is open in $V$, $M^\sfn$ is the $n$-fold groupoid from Theorem \ref{th:Usfn}.
\end{theorem}

\begin{proof}
Equivalence of (1) and (2) follows by induction from Theorem \ref{th:topcal}, and existence of $M^\sfn$ follows, by
the same principles, from Theorem  \ref{th:Usfn}.
\end{proof}

\begin{definition}\label{def:n-fold}
For any smooth Hausdorff manifold over $\K$, we call the $n$-fold groupoid
$M^\sfn = (M^{A;\sfn})_{A \in \cP(\sfn)}$ the {\em $n$-fold tangent groupoid of $M$}, or the
{\em $n$-fold magnification of $M$}. Note that each vertex set $M^{A;\sfn}$ is again a smooth manifold.
\end{definition}

\begin{theorem}[Symmetric cubic $C^n$]
Retain assumptions from the preceding theorem, and fix $\ttt \in \K^n$. 
Then for every Hausdorff manifold of class $C^n$ 
 there is an $n$-fold groupoid $M^\sfn_\ttt$ over $M$  such that,
\begin{itemize}
\item
when $M=U$ is open in $V$, $M^\sfn_\ttt$ is the $n$-fold groupoid from Theorem \ref{th:Usfn},
\item
when $\ttt = (0,\ldots,0)$, $M^\sfn_\ttt$ agrees with the $n$-fold tangent bundle $T^n M$, 
\item when
$\ttt \in (\K^\times)^n$, then $M^\sfn_\ttt$ is isomorphic to the $n$-fold pair groupoid $\pG^\sfn(M)$.
\end{itemize}
Every  $C^n$-map $f:M\to N$ induces a morphism of $n$-fold groupoids $f^\sfn_\ttt :M_\ttt^\sfn \to N_\ttt^\sfn$.
\end{theorem}

\begin{proof}
As above,
by induction, using Theorem  \ref{th:anchor}.
\end{proof}

\nin
A major difference between full cubic and symmetric cubic is that, in the latter case, we have the following
result (which fails in the full cubic case!) 

\begin{theorem}[The generalized Schwarz Theorem]\label{th:Schwarz}
For every permutation $\sigma \in {\mathfrak S}_n$, there is a natural isomorphism of
$n$-fold groupoids
$$
U_{(t_1,\ldots,t_n)}^\sfn \to U_{(t_{\sigma(1)},\ldots,t_{\sigma(n)} )}^\sfn ,
$$
inducing, for every Hausdorff $\K$-manifold $M$, a natural isomorphism 
$$
\tilde \sigma: M_{(t_1,\ldots,t_n)}^\sfn \to M _{(t_{\sigma(1)},\ldots,t_{\sigma(n)} ) }^\sfn \,  .
$$
In particular, when $\ttt = (t,\ldots,t)$ with $t\in \K$, the symmetric group ${\mathfrak S}_n$ acts
by automorphisms on $M^\sfn_\ttt$
(by definition, this means that $M^\sfn_\ttt$ is {\em edge-symmetric}).
For $t=0$, this action induces the natural action of ${\mathfrak S}_n$ on $T^n M$, as considered in \cite{Be08}, and corresponding to
the classical {\em Schwarz's theorem}.
\end{theorem}

\begin{proof}
For $n=2$, the symmetric iteration procedure is related to the ``full'' iteration procedure by letting $t_{12}=0$ in
equation  (\ref{eqn:4.1}). Thus we get
\begin{align*}
U^{[2]}_{(t_1,t_2)} &= \Bigl\{ 
(v_0,v_1,v_2,v_{12})\in V^4 \, \mid \,  \begin{matrix} v_0 \in U, \,\,  v_0 + t_1 v_1 \in U , \, \,v_0 + t_2 v_2 \in U, \\
v_0 + t_1 v_1 + t_2 v_2 + t_1 t_2 v_{12} \in U \end{matrix} \Bigr\} ,
\\
f^{[2]}(\vvv,t_1,t_2) & =\frac{ f(v_0 + t_1 v_1 + t_2 v_2 + t_1 t_2 v_{12} ) - f(v_0 + t_1 v_1) 
-f(v_0 +t_2 v_2) + f(v_0)}{t_1 t_2} .
\end{align*}
(In the latter formula, we assume that $t_1$ and $t_2$ are invertible scalars;
see \cite{Be15b} for a similar formula for $f^{[n]}(\vvv,\ttt)$ with general $n\in \N$).
From these formulae, it is immediately read off that the {\em flip} induced by the transposition $(12)$
is an automorphism from 
$U^\sftwo_{(t_1,t_2)}$ onto $U^\sftwo_{(t_2,t_1)}$ commuting with $f^\sftwo$.
By the ``density principe'' \ref{la:density}, this still holds for all $t_1,t_2 \in \K$, and by the chain rule, it
carryies over to the manifold level.
For general $n$, the claim now follows by straightforward induction.
Finally, note that the above proof is nothing but the proof of Schwarz's Theorem from \cite{BGN04}, in disguise.
\end{proof} 

\nin
Comparing with the ``full'' formula (\ref{eqn:4.1}), one sees that the full double groupoid
$U^\sett{2}$ is {\em not} edge symmetric, and that its explicit description may become quite messy. 
In the sequel, we will have a closer look at symmetric cubic calculus.

\subsection{The scalar extension viewpoint.}
 For understanding the structure of symmetric cubic calculus, it is extremely useful to view $M^\sfn_\ttt$
as the {\em scalar extension of $M$ from $\K$ to $\K_\ttt$}.
Again, the starting point is Lemma \ref{la:principle}:

\begin{lemma}\label{la:firstderivedring}
Applying the functor $\{1\}_t$ to the ring $(\K,+,\cdot,0,1)$, we get a commutative unital ring
$\K^\sfone_t$, together with two ring morphisms onto $\K$.
 This ring is isomorphic to the truncated polynomial ring $\K[X]/(X^2 - tX)$ with its two natural projections
 onto $\K = \K[X]/(X)$ and $\K=\K[X]/(X-t)$. 
\end{lemma}

\begin{proof}
The first statement follows from Lemma \ref{la:principle}.
To get the ``model'', denote
 multiplication by $m:\K \times \K \to \K$, $(x,y) \mapsto xy$, and let's compute $m^\sett{1}$ explicitly:
\begin{equation}
m^{[1]}((x,y),(u,v),t) = \frac{ (x+tu)(y+tv) - xy}{t} = uy + xv + tuv ,
\end{equation}
and for the addition map: $a^{[1]}((x,y),(u,v),t) = \frac{ (x+tu) + (y+tv) - (x+y)}{t} = u + v $, whence
$\K_t = \K^2$ with multiplication and addition given by 
$(x,u) \cdot (y,v) = (xy,xv + uy + tuv)$ and
$(x,u)+(y,v) = (x+y,u+v)$.
Put differently,
\begin{equation}
\K_t = \K^2 =
 \K 1 \oplus \K e, \,  e^2 = t e, \quad \mbox{ whence } \quad
 \K_t \cong \K[X]/ (X^2 - tX) .
 \end{equation}
 By general argments, or
by direct computation, it may be proved that the source $\alpha(u + ev)=u$ and the target
$\beta(u+ev)=u+tv$ and the unit map $1(x)=x+0e$
are indeed ring homomorphisms.
\end{proof}

\nin
Note also that, as rings,  in the special cases $t=0$ and $t=1$, we get
\begin{align*}
\K_0^\sfone &\cong \K[X]/(X^2) \mbox{ (``dual numbers'') }, \\
\K_1^\sfone  & \cong \K[X]/(X^2 - X) \cong \K[X]/(X) \times \K[X]/(X-1) = \K \times \K .
\end{align*}
Again, we can iterate constructions by induction. 
The elements $t$ and $e$ from above will be denoted $t_1$ and $e_1$, and next we adjoin another element $e_2$ such that
$e_2^2 = t_2 e_2$. This gives us  a square of rings and (pairs of) ring homomorphisms
\begin{equation}\label{eqn:square}
\xymatrix{
    \K \oplus \K e_{1}\oplus \K e_2 \oplus \K e_{12}   \ar@<-0.5ex>[d]     \ar@<0.5ex>[d]  
    \ar@<-0.5ex>[rr]     \ar@<0.5ex>[rr]  & &    \K \oplus \K e_1    \ar@<-0.5ex>[d]
     \ar@<0.5ex>[d]     \\
   \K \oplus \K e_2    \ar@<-0.5ex>[rr]
     \ar@<0.5ex>[rr]    &  &     \K   } 
\end{equation}
with relations $e_1^2 = t_1 e_1$, $e_2^2 = t_2 e_2$, $e_{12}= e_1 e_2$, whence
$e_{12}^2 = t_1 t_2 e_{12}$.
In terms of truncated polynomial rings, the preceding diagram is isomorphic to
$$
\xymatrix{
    \K[X_1,X_2]/ (X_1^2 - t_1 X_1, X_2^2 - t_2 X_2)  \ar@<-0.5ex>[d]     \ar@<0.5ex>[d]  
    \ar@<-0.5ex>[rr]     \ar@<0.5ex>[rr]  & &    \K[X_1]/(X_1^2 - t_1 X_1)   \ar@<-0.5ex>[d]
     \ar@<0.5ex>[d]     \\
   \K [X_2]/(X_2^2 - t_2 X_2  )  \ar@<-0.5ex>[rr]
     \ar@<0.5ex>[rr]    &  &     \K   } 
$$
with its natural projections and injections.
Note that there is a natural ring isomorpism, the {\em flip}, exchanging $X_1$ and $X_2$ and $t_1$ and $t_2$ (as predicted by
Th.\  \ \ref{th:Schwarz})
$$
\tau : \K \oplus \K e_{1}\oplus \K e_2 \oplus \K e_{12} \to  \K \oplus \K e_{1}\oplus \K e_2 \oplus \K e_{12}, \quad
\tau(e_1) = e_2, \tau(e_2)= e_1
$$
For general $n$, we get a hypercube of rings and ring homomorphisms that
can be described by a  $\K$-basis $(e_A)_{A \in \cP(\sfn)}$, and relations as follows:
 for each $\ttt \in \K^n$ and  $A \in \cP(\sfn)$, let
$$
t_A := \prod_{i\in A} t_i , \quad t_\emptyset := 1 . 
$$
For a vertex $C$ of the natural hypercube $\cP(\sfn)$, we define
$\K_{\ttt}^{C;\sfn}$ to be  the  free $\K$-module of rank
$2^{\vert C \vert}$, with $\K$-basis $(e_A)_{A \in \cP(C)}$, and ring structure defined by relations 
\begin{equation*}
\K_{\ttt}^{C;\sfn} = \bigoplus_{A \in \cP(C)} \K e_A, \qquad
e_A \cdot e_B = t_{A\cap B}  \cdot e_{A \cup B} 
\end{equation*}
(in particular, $e_A \cdot e_B = e_{A \cup B}$ if $A \cup B = \emptyset$).
Source and target maps corresponding to an edge $B\subset C$ with $C = B \cup \{ k \}$ are defined by
$\alpha , \beta : \K_{\ttt}^{C;\sfn} \to \K_{\ttt}^{B;\sfn} $,  where
\begin{align*}
\alpha(\sum_{A \subset C} v_A e_A ) &= \sum_{A \subset B} v_A e_A,\\
\beta(\sum_{A \subset C}  v_A e_A ) &= \sum_{A \subset B} v_A e_A + t_k
\sum_{A \subset B} v_{A \cup \{ k \}} e_{A \cup \{ k \}} .
\end{align*}
Then the hypercube of rings $(\K_{\ttt}^{C;\sfn})_{C \in \cP(\sfn)}$ with its source and target morphisms arises by $n$-fold
iteration of the construction from Lemma \ref{la:firstderivedring}. 
There is also a hypercube of {\em natural inclusions} (the unit sections from the groupoid setting), since
an inclusion $C \subset D$ induces an inclusion $\cP(C) \subset \cP(D)$.
The following special cases deserve attention:
if $t_i=1$ for all $i$, we get the {\em idempotent ring} with relation $e_A \cdot e_B = e_{A \cup B}$ for all
$A,B \in \cP(C)$, which in fact is isomorphic to a direct product of $2^{\vert C \vert}$ copies of $\K$.
In the ``most degenerate'' case $t_i=0$ for all $i$, we get the {\em $n$-th order tangent ring} $T^n \K$ used extensively in \cite{Be08, BeS14},
with relation $e_A \cdot e_B = 0$ whenever $A \cap B \not= \emptyset$. 
This is a hypercube of  {\em Weil algebras} in the sense of
\cite{KMS93, BeS14} 
(the ideal, kernel of $\alpha$ or $\beta$, is nilpotent), whereas for invertible $t_i$ the algebras are never Weil algebras. 
Therefore we propose the following concept,
replacing  the notion of Weil algebra in our context: 

\begin{definition}
A {\em cubic ring} (of order $n$) $\bA$ is given by a family of rings and ring morphisms: 
 for each vertex $A$ of the hypercube $\cP(\sfn)$,  there is a (unital, commutative) ring (``vertex ring'') $\bA^A$,
and, for every edge $(B,A)$ of the hypercube,  two ring morphisms (``edge projections'') 
$\alpha^{B,A},\beta^{B,A}: \bA^A \rightrightarrows \bA^B$,
and a ring morphism section $1^{A,B}:\bA^B \to \bA^B$ of both of them, 
such that for each face of the hypercube, the obvious diagrams of morphisms commute.
\end{definition}

\nin
The preceding discussion is summarized by

\begin{theorem}
If $(\K,m,a)$ is a good topological ring and $\ttt \in \K^n$, then $\bA :=\K_\ttt := (\K^{A,\sfn}_{\ttt})_{A \in \cP(\sfn)}$
is a cubic ring.
Every vertex ring is again a good topological ring.
\end{theorem}

\nin
One may say that the accent is shifted from an individual algebraic property (nilpotency of the ideal) to a ``social'' property of
algebras: algebras live in families structured by cubes; ideals live in families of two kinds (source and target kernels) and parametrized
by continuous parameters $\ttt$. 
Moreover, this family carries the structure of an $n$-fold groupoid, which is not mentioned in the definition of cubic ring.
The following ``main theorem'' says that this rich social structure encodes general structure of ``conceptual calculus on manifolds'':
the groupoids  $M^\sfn_\ttt$ can be interpreted as {\em scalar
extensions of $M$ from $\K$ to $\K_\ttt^\sfn$}.

\begin{theorem}[The scalar extension theorem]\label{th:scalarextension}
If $M$ is a smooth Hausdorff manifold over the good topological ring $\K$, then, for all $n \in \N$,
$\ttt \in \K^n$ and $A \in \cP(\sfn)$, the manifold
$M_{\ttt}^{A;\sfn}$ is {\em smooth over the ring $\K^{A;\sfn}_{\ttt}$}, and 
if $f:M \to N$ is smooth over $\K$, then
$f^{A;\sfn}_{\ttt}$ is {\em smooth over the ring $\K^{A;\sfn}_{\ttt}$}.
\end{theorem}

\begin{proof}
The arguments, again by induction
 based on Lemma \ref{la:principle}, are verbatim the same as those proving \cite{Be08}, Theorems
6.2 and 7.2 (which concern the case $\ttt = (0,\ldots,0)$ and $M^{\sfn;\sfn}_{0} = T^n M$, the $n$-th order tangent
bundle). 
\end{proof}

\subsection{Consequences}
The preceding theorem is a central result:
as said in the introduction to \cite{Be08}, that work arose from working out all consequences of Theorems
6.2 and 7.2 from loc.\ cit. 
In a similar way, the consequences of Theorem \ref{th:scalarextension} might also fill a whole book. 
Therefore I will stop here a description of the formal theory,  and try instead to give an overview over some topics that could be
part of the contents of that book.
The main strands of \cite{Be08}, approached via the scalar extension point of view, and interwoven with each other, are
{\em connection theory} and {\em Lie theory}.  
I will give some comments on these two topics, from the point of view of ``Lie calculus'' as advocated here.
Before doing so, I'd like to stress once again that the theory will cover both the infinitesimal and the local, or even global, description
differential geometric objects.  This is new even in the classical setting of real, finite-dimensional
manifolds: the object encoding infinitesimal geometry, the tangent bundle $TM$, and the one encoding local or global
information, the pair groupoid $\pG(M)$, are both classical, but -- apart from Connes' tangent groupoid (cf.\ comments
on def.\ \ref{def:tangentgroupoid}) -- there has been no theory putting them into a common framework.

\subsubsection{Lie Theory}\label{ssec:Lietheory}
The heart of Lie Theory is the {\em Lie group-Lie algebra correspondence}.
In \cite{Be08}, several independent definitions of the Lie bracket of a Lie group $G$ are given: one may start with the
{\em Lie bracket of vector fields}, and use it to define the Lie algebra $\g$ via left- or right invariant vector fields, or go the other
way round and define the Lie bracket via
a {\em group commutator} $[g,h]=ghg^{-1}h^{-1}$ in the {\em second} tangent group $TTG$. 
In both cases, the stage is set by {\em second order calculus}: at first order, we do not yet ``see'' the group structure of $G$, but
only its first approximation which is in fact given by the canonical groupoid law of the underlying space.
To prove the Jacobi identity, computations involve {\em third order} calculus.
In \cite{Be08}, this is pushed further to analyze the group structure
of all higher order tangent bundles $T^n G$ (see also \cite{V13} for the structure of the jet bundle $J^nG$).

\ssk
To a large extent, all  this perfectly carries over to the groups $T^n G$ replaced by $G^\sfn_\ttt$. 
One of the main ingredients from the infinitesimal theory,
the {\em vertical bundle} $VM$ sitting inside $TTM$ and forming a sequence (cf.\ \cite{Be08}, eqn. (7.8))
\begin{equation}\label{eqn:vertical}
\begin{matrix}
TM \cong VM & \to  & TTM & \to & (TM \times_M TM) \, ,
\end{matrix}
\end{equation}
is generalized and ``conceptualized'' by the {\em core structure}: the {\em core of a double groupoid} (cf.\ \cite{BrMa92}) has 
a higher dimensional analog which has a nice description in terms of our cubic rings $\K^\sfn_\ttt$:

\begin{definition}
For subsets $\emptyset \not= B \subset C \subset \sfn$, consider the $(\vert C \vert - \vert B \vert)$-hypercube
$$
\cP_B^C(\sfn):= \{ A \in\cP(\sfn) \mid \, B \subset A \subset C  \} 
$$
which corresponds to the hypercube of ideals in the  vertex algebra $\K^{C,\sfn}_\ttt$ given by
$$
I_B^C(\K^\sfn) :=\bigoplus_{A \in \cP_B^C(\sfn)} \K e_A .
$$
For fixed $B$,
the corresponding {\em $B$-core cube} is the  cubic ring $(\K \oplus I_B^C(\K^\sfn))_{C \in \cP_B^\sfn(\sfn)}$.
\end{definition}

\nin
The core cubes globalize to the manifold level, 
and thus define analogs of the sequence (\ref{eqn:vertical}), which can be used as ingredient to define
a version of the Lie bracket on the bundles $G^\sfn_\ttt$. 
Of course, it shall also be used to give a general and clean construction of the {\em Lie algebroid of a Lie
groupoid} in the present context (cf.\ \cite{SW15} for this item). 


\subsubsection{Connections}\label{ssec:connections}
Lie theory can be considered as part of {\em connection theory} -- but the converse could probably be justified as well, and
therefore I prefer to discuss these two topics independently of each other.
Indeed, there is a beautiful, but not very well known, approach to connections via {\em loop theory}, developed by
L.\ Sabinin in a long series of papers (cf.\ his monograph \cite{Sa99}).
This theory is algebraic in nature, and hence perfectly
 suited to be adapted to our framework. As Sabinin puts it (loc.\ cit., p. 5):
{\em Since we have reformulated the notion of an affine connection in a purely algebraic language, it is possible now to treat
such a construction over any field (finite if desired)... 
Naturally, the complete construction needs some non-ordinary calculus to be elaborated.}
I do think that the non-ordinary calculus he dreamt of exists now, and that
 nothing prevents us from following the plan outlined by this phrase.
Indeed,  I have been working on this topic for quite a while, and mainly for reasons of time the manuscript is not yet achieved.
To describe Sabinin's idea in a few words, adapted to the preceding notation:
when working with groupoids, one sometimes regrets that the product $\ast$ is not everywhere defined, and one 
 would like to work with some {\em everywhere defined} product.
 This is essentially what a
{\em connection on a groupoid} provides --  you just {\em have to give up associativity!} To be more precise, 
a connection on a groupoid $G$ corresponds to an everywhere defined ternary product $(a,b,c) \mapsto a \bullet_b c$ on 
$G$ extending, or ``integrating'', the not everywhere defined 
ternary groupoid product $a \ast b^{-1}\ast  c$, such that each binary product
$(a,c) \mapsto a\bullet_b c$ is a {\em loop}.
Indeed, when $M=U$ is open in a linear space $V$, then on $G = U^\sett{1}_t$ there
 is a natural ternary product of this kind, given by 
the  locally defined torsor structure
$(x,v) \bullet_{(x',v')} (x'',v'') = (x-x'+x'',v-v'+v'')$. It corresponds to the {\em canonical flat connection induced by $V$}.
This approach is very much in keeping with the one from Synthetic Differential Geometry (\cite{Ko10}), where 
{\em connections on groupoids} are defined in a similar way (retaining only the infinitesimal, not the local, information). 
For instance, if $G$ is a Lie group, then the globally defined torsor structure, and its opposite, on $G^2$ define two such  connections,
called the {\em canonical left and right connection of $G$}. 
Lie theory can be recast in this language: associativity corresponds to curvature freeness of these two connections, 
and so on. 
I believe that this algebraic approach not only is the most general possible, but also sheds new light on the {\em geometry of loops}
(in particular, their close link with {\em $3$-webs}, see \cite{AkS92, NS02, Sa99}).\footnote{ To add a personal note, 
I met Karl Strambach for the last time
 on the 50th Seminar Sophus Lie, when exposing these projects, and he was quite delighted
by the idea that these seemingly forgotten conceptions relating loops and differential geometry could be revived.}

\section{Perspectives}\label{sec:perspectives}

The preceding remarks on Lie and Connection Theory naturally lead to add some more comments on open problems and further
research topics. 

\subsection{Discrete {\sl versus} continuous}\label{ssec:discrete}
In the present text, basic definitions and results are given in the framework of topological calculus over
good topological rings (Def.\ \ref{def:CnK}), thus using topology and continuity,
whereas in \cite{Be15, Be15b}, I have put the accent on the possibility of developing the whole theory over {\em discrete} base rings,
that is, of developing a purely algebraic theory, applying, e.g., to $\K=\Z$, or even a {\em finite} ring.
Although I'm afraid the readability of these papers has suffered a bit under this extreme degree of generality, 
I do believe that in the long run this is an  important aspect: quantum theory suggests that the universe be discrete in nature, and
hence we would like to understand how calculus (one of our main tools when doing mathematical physics!) could be adapted to this
situation. The basic idea is very simple: just like, in algebra, a polynomial is a formal object, a ``space over $\K$'' will be a formal object, too,
not necessarily uniquely determined by its base set $M$, but rather by the whole bunch of information carried along by all its
``extensions'' $M^\sfn_\ttt$ for $n \in \N$, satisfying all the formal relations explained in this text. 
Likewise, a ``$\K$-smooth map'' between such objects is not necessarily determined by its underlying set-map $f:M \to M'$, but
by all its extensions $f^\sfn_\ttt$.
In topological differential calculus, the use of topology serves
  to store all this information in the base space $M$ and in the base map $f$ -- 
necessarily, we need an infinite ring (and an infinite unit group $\K^\times$) in order to extract this information, via the 
``density principle'' \ref{la:density}.
In the purely algebraic theory, this infinite information is explicitly given in an ``attached file'', allowing the base objects
$M$ and $f$ to be possibly finite.

\ssk
To a certain extent, this approach works very well, but of course it has its limits. These limits, in turn, may be starting points for new problems
and new challenges: for instance, we must first understand the formal properties of the local connections defined by Sabinin (see above,
\ref{ssec:connections}); geodesics and the {\em exponential jet} (\cite{Be08}, Chapter VI) cannot be defined by integrating 
differential equations, so we have to understand their formal structure; and it is quite a challenge to reformulate 
notions and results involving {\em volume}:  volume is a local or global property, which can make sense in a discrete space, but it
is not clear how this should be related to the infinitesimal theory.

\subsection{Full cubic calculus, positive characteristics, and the scaloid}
Understanding the relation between ``full cubic'' and ``symmetric cubic'' calculus (Section \ref{ssec:full-vs-symmetric}) 
becomes particularly important
in the case of {\em positive characteristic}, and for finite base rings. This can be seen by remembering that the classical
 {\em Taylor formula} involves terms $\frac{1}{k!}$, and hence does not carry over to the case of positive characteristic.
 However, the {\em general Taylor formula} from \cite{BGN04} does make sense over any base ring.
A closer inspection shows that this formula really belongs to {\em full cubic} calculus,  
and more precisely to the ``non-symmetric'' aspect of full calculus, which has been christianed in \cite{Be13}
{\em simplicial differential calculus}.
Thus, although symmetric cubic calculus can be defined over any base ring, it is sort of ``incomplete'' in certain cases (such as
finite rings). I believe that understanding what is going on here is important also for the general case. 

\ssk
Fortunately, all the specific difficulties of full cubic calculus concentrate in a single algebraic object, the
{\em scaloid} (cf.\ \cite{Be15b}): 
let us call {\em naked point} and denote by $0$ the zero-subspace of the zero-$\K$-module $\{ 0 \}$.
By definition, the {\em scaloid} is the family of $n$-fold groupoids $0^\sfn$, for $n\in \N$.
One should not think that $0^\sfn$ be trivial:
already $0^{[1]} = \K$ is not a trivial set, although 
$0^\sfone = (0^{[1]}, \K) = (\K,\K)$ is indeed trivial {\em as a groupoid}. 
But  $0^\sftwo$ is a non-trivial gropoid, and this argument shows that the theory of
$0^\sfn$ and of $\K^{\mathsf n-1}$ is essentially the same.
The abstract reason for the importance of $0^\sfn$ is that usual cartesian products should be seen as
fibered product over $0$, in formulas, $A \times B = A\times_0 B$, and our ``rule $\sfn$'' is compatible with
fibered products, rather than with plain cartesian products:
$(A \times_M B)^\sfn = A^\sfn \times_{M^\sfn} B^\sfn$, making it natural that $0^\sfn$ appears whenever we
work with cartesian products.
 Personally, I like to think of the scaloid as some kind of ``elementary particle'' that remained unobserved in the usual theories
-- such theories  are symmetric cubic in nature, and the symmetric cubic groupoid $0_\ttt^\sfn$ is indeed trivial as set
and as groupoid.


\subsection{General spaces, and relation with SDG}\label{ssec:SDG}
In \cite{MR91},  p. 1--3,  Moerdijk and Ryes
 give three main reasons for generalizing the ``ordinary'' theory by Synthetic Differential Geometry (SDG)
 (cf.\ also \cite{Be08}, Appendix G):
 \begin{enumerate}
\item  the category of smooth manifolds is not cartesian closed (spaces of mappings between manifolds are not always manifolds),
\item  the lack of finite inverse limits in the category of manifolds (in particular, manifolds can not have ``singularities''),
\item  the absence of a convenient language to deal explicitly and directly with structures in the ``infinitely small''.
\end{enumerate}
I claim that the theory started here allows to achieve the same goals by different means, and this in much greater generality
since models of SDG all use the real numbers in one way or another, whereas our theory does not use them. 
Indeed, a natural answer to
(3) is given by the scalar extension viewpoint explained above; as to (1) and (2), we have to
go beyond the framework of smooth manifolds. In our theory, there is a natural way to do this: 
kernels of morphisms of higher order groupoids $M^\sfn$, and quotients of them, are again higher order groupoids,
and hence one may single out some convenient (big) category of such higher order groupoids in order to describe
more general ``spaces''. 
Such a procedure remains in the framework of classical algebra and classical set-theory, whereas SDG tries to 
achieve these goals by very different methods (topos theory, using intuitionistic logic and avoiding the law of the excluded
third). 
However, it seems very well possible to combine the methods used here with those used in SDG in order to 
develop some kind of ``SDG over general base fields and -rings''.

\subsection{Non-commutative base rings, supersymmetry; left versus right}
It is intriguing to observe that the first order theory works perfectly well over arbitrary, possibly 
{\em non-commutative} base rings $\K$;
only at second and higher order level, commutativity of $\K$ is needed (cf.\ \cite{Be15}). 
So, what exactly is the obstruction for defining ``conceptual calculus over non-commutative base rings''? 
I don't know the answer, and very likely there is no theory admitting completely general non-commutative base rings.
However, I have the impression that {\em super-commutative} rings should be admissible: 
there should be a common framework including both
 ``conceptual super-calculus'' and ``conceptual calculus''.
However, in spite of several tries, I'm not yet sure about the form that such a theory should take.
My feeling is that super-calculus should arise from taking account of the fact that the definition of a groupoid is completely symmetric
in source $\alpha$ and target $\beta$: a groupoid and its opposite groupoid have, in principle, ``equal status''. 
To a certain extent, conceptual calculus is also symmetric in source $\alpha$ and target $\beta$.
And yet this symmetry must be broken at a certain point -- it is not quite clear when this point is reached, but 
it should be the bifurcation point  where ``usual'' and ``super'' calculus separate. 
Of course, our formulae somehow ``prefer'' the source $\alpha$ (having a very simple expression, 
whereas the one for $\beta$ 
in cubic calculus is extremely complicated; cf. \cite{Be15}), but that may be some accidental and not intrinsic feature.
It rather seems to me that this symmetry is not broken until we really use {\em mappings} as a tool, and work with the ``usual'' conventions
about them: they are binary relations having certain properties, and which their opposite relations do in general not have
(cf.\ example \ref{ex:bisections}). 
Thus the symmetry might possibly be restored by working with general binary relations, instead of
mappings: calculus and super-calculus might be different aspects of a single ``relational calculus''.
This may be less crazy than it sounds: it just would mean to take the groupoid point of view seriously.

\end{document}